\documentclass[11pt]{article}
\usepackage{amssymb,amsbsy,amsmath,amsfonts,amscd,amsthm}
\usepackage{fullpage}
\usepackage{abstract}
\usepackage{mathabx}
\usepackage{mathrsfs}
\usepackage{enumitem}
\numberwithin{equation}{section}
\usepackage[T1]{fontenc}
\usepackage[toc,page]{appendix}
\usepackage{url}
\usepackage{courier}
\usepackage{graphicx}
\graphicspath{{./}{figures/}{plots/}} 
\usepackage{tikz}
\usepackage{diagbox}
\usepackage{float}
\usetikzlibrary{matrix}
\usepackage{algorithm}
\usepackage{algpseudocode}
\numberwithin{algorithm}{section}
\usepackage{mathtools}
\usepackage{aliascnt}
\newtheorem{theorem}{Theorem}[section]
\newaliascnt{lemma}{theorem}
\newtheorem{lemma}[lemma]{Lemma}
\aliascntresetthe{lemma}
\newaliascnt{proposition}{theorem}
\newtheorem{proposition}[proposition]{Proposition}
\aliascntresetthe{proposition}
\newaliascnt{corollary}{theorem}
\newtheorem{corollary}[corollary]{Corollary}
\aliascntresetthe{corollary}
\newaliascnt{remark}{theorem}
\newtheorem{remark}[remark]{Remark}
\aliascntresetthe{remark}
\newaliascnt{definition}{theorem}
\newtheorem{definition}[definition]{Definition}
\aliascntresetthe{definition}
\usepackage[hidelinks]{hyperref}
\usepackage[capitalise]{cleveref}
\makeatletter
\providecommand*{\theHALG@line}{}
\renewcommand*{\theHALG@line}{\arabic{section}.\arabic{algorithm}.\arabic{ALG@line}}
\makeatother
\crefname{equation}{}{}
\Crefname{lemma}{Lemma}{Lemmas}
\Crefname{proposition}{Proposition}{Propositions}

\newcommand{\lc}{\left(}
\newcommand{\rc}{\right)}

\def\su{{\mathsf{u}}}
\def\sv{{\mathsf{v}}}

\newcommand{\eps}{\varepsilon}

\newcommand{\ux}{{\underline x}}

\newcommand{\beq}{\begin{equation}}
\newcommand{\eeq}{\end{equation}}

\newcommand{\g}{\gamma}

\newcommand{\p}{\psi}

\theoremstyle{plain}

\def\sideremark#1{\ifvmode\leavevmode\fi\vadjust{
		\vbox to0pt{\hbox to 0pt{\hskip\hsize\hskip1em
				\vbox{\hsize3cm\tiny\raggedright\pretolerance10000
					\noindent #1\hfill}\hss}\vbox to8pt{\vfil}\vss}}}

\newcommand{\D}{\Delta}
\newcommand{\Dx}{{\Delta x}}

\title{Finite difference methods for a continuous-time heterogeneous agent model with recursive utility}
\author{Yves Achdou\thanks{Universit\'{e} de Paris Cit\'{e} and Sorbonne Universit\'{e}, CNRS, Laboratoire Jacques-Louis Lions,  achdou@ljll.univ-paris-diderot.fr}, \quad Qing Tang\thanks{School of Mathematics and Physics, China University of Geosciences (Wuhan),  tangqingthomas@gmail.com}}
\date{\today}

\begin{document}
\maketitle

\begin{abstract} 
We propose, analyze and test computational methods for solving a continuous-time heterogeneous agent model with Epstein--Zin utility. Such recursive utilities allow the model to disentangle between risk aversion and intertemporal substitution. Having discretized the Hamilton--Jacobi--Bellman (HJB) equation arising in the model, we propose a Howard--Newton algorithm for the late resolution preference case, and a Howard--Tarski--Kantorovich algorithm for the early resolution preference case. We prove the convergence of the iterative algorithms. We obtain as a consequence the existence of solutions to the discretized HJB equations. In the late resolution case, we supply a priori estimates between the unique solutions of the continuous and discretized HJB equations. \\
\textbf{Keywords:} Hamilton--Jacobi--Bellman equation, Mean Field Games, recursive utility, finite difference, iterative algorithms
\end{abstract}

	
\section{Introduction}
In this work, we propose, analyze and test computational methods for solving the continuous-time heterogeneous agent (HA) models with recursive utility, recently analyzed by the authors in \cite{achdou2026continuous}.  The continuous-time formulation of the Aiyagari--Bewley--Huggett models \cite{ljungqvist2018recursive}, classical in recursive macroeconomics, and the related system of partial differential equations can be studied in the light of the mathematical theory of Mean Field Games (cf. \cite{MR2679575,achdou2021mean,MR2295621}). Such models involve a large number of \textit{ex ante} identical but \textit{ex post} heterogeneous agents in an incomplete market setting: each agent continuously optimizes consumption subject to a borrowing constraint and idiosyncratic income risk, resulting in a non-degenerate stationary wealth distribution (cf. \cite[(1.1)]{achdou2026continuous}). The Mean Field Game system consists of Hamilton--Jacobi--Bellman (HJB) equations and Fokker--Planck--Kolmogorov (FPK) equations coupled through equilibrium conditions on prices. HA models have now become indispensable for quantifying the distributional effects of fiscal and monetary policy. A finite difference method for the constant relative risk aversion (CRRA) utility case was developed in \cite{achdou2022income} and its numerical appendix. This method from \cite{achdou2022income} has then become popular in the HA macro literature, as the first step for computing stationary equilibria \cite{ahn2018inequality,bilal2023solving,fernandez2023financial,gu2024global,kaplan2018monetary}. The present paper aims at a rigorous analysis of a related finite difference method for models with recursive utility (Epstein-Zin). In particular, we propose convergent iterative algorithms for solving the discretized Hamilton-Jacobi-Bellman equations. \par
The Epstein-Zin recursive utility is defined as follows (cf. \cite{wang2016optimal}):
       \begin{equation}
         \label{EZ}
       f(c,v)=\frac{\rho}{1-\p^{-1}}\frac{c^{1-\p^{-1}}-((1-\g)v)^\theta}{((1-\g)v)^{\theta-1}},\qquad \theta=\frac{1-\p^{-1}}{1-\g},
       \end{equation}
       where $c$ is the consumption, $v$ is the value function of an optimal strategy and $\rho$ is the subjective discount rate. It is assumed that $\p$, the elasticity of intertemporal substitution (EIS) and $\g$, the risk aversion parameter, are both positive and do not take the value $1$. A key feature of the Epstein-Zin utility \eqref{EZ} is the separation between risk aversion and EIS. The time-additive separable CRRA utility is a special case of recursive utility in which $\g=\psi^{-1}$.  The attitude towards the timing of the resolution of uncertainty is pinned down by the constant $\g \p$: early (resp. late) resolution is preferred if $\g \p> (resp. <) 1$, cf. \cite{EZ1989}. With $\g=\psi^{-1}$, the agent is indifferent to the timing of uncertainty resolution.\par
The Mean Field Game system of partial differential equations for an Aiyagari model with the recursive utility \cref{EZ} is (cf. \cite[(1.7)]{achdou2026continuous}):
\begin{equation}\label{MFG}
\left\{\begin{aligned}
\quad &(i)\qquad &&0 = \max_{c\ge 0} \left\{ f(c,v_j) +(r^*x+y_j - c)Dv_j(x) \right\}+\lambda_j(v_{\bar \jmath}(x)-v_j(x)),\\
&&&c_{j}^*(x)=\mathop{\arg \max}\limits_{c\ge 0} \left\{ f(c,v_j) +(r^*x+y_j - c)Dv_j(x) \right\},\\
&(ii)&&- \frac{\partial}{\partial x} \left[ (r^*x+y_j - c_{j}^*(x)) g_j(x) \right]+\lambda_{\bar \jmath}g_{\bar \jmath}(x) -\lambda_jg_j(x)=0,\\
&&&\int_{x\geq \ux} g_1(x)dx+\int_{x\geq \ux} g_2(x)dx+\sum_{j\in \{1,2\}}\mu_j=1,\\
&(iii_A)\qquad &&r^*=A\alpha \left(\frac{\mathcal{K}[r^*]}{N}\right)^{\alpha-1}-\delta=A\alpha \left(\frac{K[m^{(r^*)}]}{N}\right)^{\alpha-1}-\delta. 
\end{aligned}\right.
\end{equation}
The HJB equations \cref{MFG} $(i)$, one for each $j$ ($j\in \{1,2\}$ and $\bar\jmath=3-j$) are coupled through the  switching terms $\lambda_j(v_{\bar\jmath} - v_j)$, encoding the transitions between income states with Poisson rates $\lambda_j$. The state constraint $x \geq \ux$ (a borrowing limit) is imposed to prevent negative savings at $\ux$: $r^*\ux+y_j - c^*_j(\ux)\geq 0$. Each of the stationary FPK equations \cref{MFG} ($ii$) describes the density $g_j$ of the invariant measure $m_j$ associated with the agents with income $y_j$. The measure $m_j$ possibly exhibits a Dirac mass at $\ux$, weighted by $\mu_j$. The equilibrium interest rate $r^*$ is determined by the fixed-point condition \cref{MFG} $(iii_A)$, where the aggregate capital stock $\mathcal{K}[r^*]=K[m^{(r^*)}]=\sum_j\int_{x\geq \ux}xdm^{(r^*)}_j$ is itself a functional of the invariant measure (which depends on $r^*$), and thus closes the model. Although the {\textit{Mean Field}} coupling \cref{MFG} $(iii_A)$ is particular, the analysis developed below applies to a wide range of HA models.
 \par
 To solve the state constrained HJB equation ($i$) with a given $r>0$, we use monotone (upwind) finite difference schemes, as in \cite{MR2679575,achdou2022income,camilli2025learning,lauriere2023policy,tang2023learning}. The main difficulties in this model are: firstly, the Hamiltonian takes infinite values for negative momentum; secondly (this is specific to the recursive utility) the dependence of $f$ on $v$. The second difficulty implies that one should use iterative algorithms that differ from the standard Howard algorithms used for example in \cite{achdou2022income,camilli2025learning}. 
A main discovery of the present paper is that different strategies are suitable depending on the timing preference in the models. In the late resolution case $\theta\geq 1$, the system satisfies a comparison principle, and we use a Howard--Newton algorithm to solve for the unique solution. In this method, the policy evaluation step consists of an inner loop of Newton iterations. In the early resolution case $0<\theta< 1$, we consider a solution to the HJB equation as a fixed point of a suitably defined map $\Gamma$, and take advantage of the monotonicity and invariance properties of $\Gamma$ to design a Howard-Tarski-Kantorovich algorithm. We have decided to name the method so because the design of the outer loop is reminiscent of Tarski--Kantorovich fixed point theorem and its proof in \cite[Theorem I, p.68]{kantorovitch1939method}.  We prove the convergence of these iterative algorithms, and this also implies the existence of solutions to the discretized HJB equations. The Fokker--Planck--Kolmogorov equations will then be discretized with a finite difference scheme, exploiting the duality structure of \cref{MFG} (cf. \cite{achdou2022income}). \par
The methodology for studying the case $0<\theta< 1$ in the present paper is related to the recent literature about dynamic programming on ordered spaces and Koopman operators, e.g. \cite{bloise2024not,Sargent_Stachurski_2025,Sargent2025}. In these works, the Tarski-Kantorovich fixed point theorem have been used to study discrete time Bellman equations with recursive utility. In the present paper, we show that the barriers and invariance properties, essential for using this method, can be constructed for the upwind schemes that we propose. \par
We analyze the convergence of the numerical solution to the constrained viscosity solution of the HJB equation, in the case of $\theta\geq 1$. Recently, a Semi-Lagrangian method was proposed in \cite{SL_HA} to study continuous-time HA models with CRRA utility, with convergence analysis based on the Barles-Souganidis half-relaxed limits method. In the present paper, we obtain error estimates with a technique that  does not require doubling the variables, thanks to the strong regularity of the solutions to \cref{MFG} $(i)$. 
\section{Preliminaries}
The assumptions that follow will be made in the whole paper:
\beq\label{asp1}
\g>1,\,\,0<\p<1,\quad \rho> r>0,\quad y_2>y_1>0,\quad \ux\leq 0,\quad \rho \ux+y_1>0.
\eeq
It has been proved in \cite{achdou2026continuous} that the equilibrium interest rate $r^*$ is smaller than $\rho$, if there exists a solution to \cref{MFG}. This justifies the bounds on $r$ in \cref{asp1}. With the constant $b$ defined in \cref{tab1} below, \cref{asp1} yields $r<b<\rho$, $rx+y_2>rx+y_1>0$ and $b(x+y_1/r)>0$. Therefore, the functions in \cref{Eq: sub super} below are well defined. \par  
With Epstein-Zin utility \eqref{EZ}, we can rewrite the HJB equation \cref{MFG} (i) 
	\beq\label{HJB}
	\begin{aligned}
	\frac{\rho}{\theta} v_j(x)=	H(x,y_j,v_j,Dv_j)+\lambda_j(v_{\bar \jmath}(x)-v_j(x)),
	\end{aligned}
	\eeq
	where we use the notation
	\beq\label{EZ flow}
	f(c,v_j)=\mathcal{F}(c,v_j)-\frac{\rho}{\theta} v_j,\qquad \mathcal{F}(c,v_j)=\frac{\rho}{1-\p^{-1}}\frac{c^{1-\p^{-1}}}{((1-\g)v_j)^{\theta-1}},
	\eeq
\beq\label{Eq: H from F}
 H(x,y_j,v_j,p):=\max_{c\ge 0} \left\{ \mathcal{F}(c,v_j) +(rx +y_j- c)p\right\}.
\eeq
From the first order necessary optimality condition, the Hamiltonian in \Cref{Eq: H from F} is
	\begin{equation}\label{def H}
  H(x,y,v,p)=\left\{
      \begin{aligned}
     \quad  (rx +y)p+\frac{\rho^{\p}}{\p-1} p^{1-\p}((1-\g)v)^{\frac{1-\g \p}{1-\g}},\quad \quad & \text{if }\quad p\geq 0,\\
       +\infty, \quad \quad  & \text{if }\quad p< 0.
     \end{aligned}
   \right.
\end{equation}
We observe that $H(x,y,v,p)$ defined in \Cref{def H} is strictly convex in $p$ for fixed $(x,y,v)$ with $p>0$, and that
\beq\label{H min}
\min_{p>0}H(x,y,v,p)=\rho\frac{(rx+y)^{1-\p^{-1}}}{1-\p^{-1}}\big((1-\g)v\big)^{\frac{\p^{-1}-\g}{1-\g}}.
\eeq
Moreover, since $rx+y>0$ and $p^{1-\p}$ is sublinear, we infer from \Cref{def H} the coercivity property
\beq\label{H coerc}
H(x,y,v,p)\to +\infty \quad \text{when}\quad  p\to +\infty.
\eeq
The optimal consumption (away from the borrowing limit) is given by the first order necessary optimality condition, whenever $Dv_j >0$, 
	\beq\label{c optimal}
	c_j=\mathop{\arg\max}\limits_{c\geq 0} \left\{ \mathcal{F}(c,v_j) -cDv_j \right\}=\rho^{\p}(Dv_j)^{-\p}((1-\g)v_j)^{\frac{1-\g \p}{1-\g}}.
	\eeq

We summarize the notation in \Cref{tab1}. 
\begin{table}[htbp]
    \centering
    \caption{Symbols}
    \label{tab1}
    \begin{tabular}{l c c}
        \hline
        Subjective discount rate & $\rho>0$ \\
        Risk aversion  & $\gamma>1$  \\
        Elasticity of intertemporal substitution (EIS) &$0<\psi<1$\\
        A parameter arising in Epstein-Zin utility &$\theta:=\frac{1-\p^{-1}}{1-\g}>0$\\
        A convenient parameter &$b:= \rho\left[\frac{r+\p(\rho-r)}{\rho}\right]^{\frac{1}{1-\p}}$\\
        Aggregator & $f(c,v)$\\
        Modified aggregator & $ \mathcal{F}(c,v)=f(c,v)+\frac{\rho}{\theta}v$\\
        Borrowing limit & $\ux$ \\
       \hline
    \end{tabular}
\end{table}\par
For any $c>0$, if $\theta > 1$ (resp. $0<\theta<1$) then $\mathcal{F}(c,v)$ is decreasing (resp. increasing) in $v$, i.e. 
\beq\label{F monotone}
\mathcal{F}_v(c,v)< 0\quad (resp. \,\,\,\mathcal{F}_v(c,v)>0).
\eeq
If $\theta > 1$, then $f(c,v)$ and $\mathcal{F}(c,v)$ are jointly concave in $(c,v)\in (0,+\infty)\times (-\infty,0)$.\par
We denote by $H_v(x,y,v,p)$ and $H_{vv}(x,y,v,p)$ the first and second order derivatives of $H(x,y,v,p)$ with respect to $v$. We use $H_{vp}(x,y,v,p)$ to denote the cross second order derivative of $H(x,y,v,p)$ with respect to $v$ and $p$.  Straightforward computations lead to:
\beq\label{Eq:Hv>0}
\begin{aligned}
\quad &H_v(x,y,v,p)<0,\,\,  H_{vp}(x,y,v,p)< 0,\,\,  H_{vv}(x,y,v,p)< 0\,\, {\text{if}}\,\,\theta > 1,\\
&H_v(x,y,v,p)> 0,\,\,  H_{vp}(x,y,v,p)>0,\,\,  H_{vv}(x,y,v,p)>0\,\, {\text{if}}\,\,0<\theta<1.
\end{aligned}
\eeq
Therefore, $H(x,y,v,p)$ is strictly concave (resp. convex) in the $v$ variable if $\theta > 1$ (resp. $0<\theta<1$). \par
\begin{definition}\hfill\label{def vis sol}
	\begin{itemize}
		\item[1.] A continuous function $v=(v_1,v_2)$ is said to be a viscosity subsolution of \eqref{HJB} at $x$, if whenever $\varphi$ is a smooth function and $v_j-\varphi$ has a local maximum at $x$, then
	\[\frac{\rho}{\theta} v_j(x)\le H(x,y_j,v_j(x),D\varphi(x) )+\lambda_j (v_{\bar \jmath}(x)-v_j(x)).\]	
		\item[2.] A continuous function $v=(v_1,v_2)$ is said to be a viscosity supersolution of \eqref{HJB} at $x$, if whenever $\varphi$ is a smooth function and $v_j-\varphi$ has a local  minimum at $x$, then
	\[\frac{\rho}{\theta} v_j(x)\ge H(x,y_j,v_j(x),D\varphi(x) )+\lambda_j (v_{\bar \jmath}(x)-v_j(x)).\]	
	\item[3.] 	A continuous function $v$ is said to be a constrained viscosity solution to  system \eqref{HJB} if $v$ is a viscosity supersolution in $(\ux,\infty)$ and a viscosity subsolution in $[\ux,\infty)$.
	\end{itemize}
\end{definition}
The sub- and supersolutions proposed below will play an important role in what follows. The following result has been proved in \cite[Proposition 3.5]{achdou2026continuous} for the case $\theta \geq 1$. We observe that the same proof holds for $0<\theta<1$.  
\begin{proposition}\label{Prop:sub-super}
With the constant $b$ defined in \cref{tab1}, we consider the functions 
\beq\label{Eq: sub super}
\begin{aligned}
\check{\su}_1(x)=\check{\su}_2(x)=\frac{(rx+y_1)^{1-\gamma}}{1-\gamma},\quad
 \check{\sv}_1(x)=\check{\sv}_2(x)= \frac{(b(x+y_2/r))^{1-\gamma}}{1-\gamma}.
\end{aligned}
\eeq
 The pairs $(\check{\su}_1,\check{\su}_2)$ and $ (\check{\sv}_1,\check{\sv}_2)$ are respectively a subsolution of \eqref{HJB} in $[\ux,+\infty)$ and a supersolution of \eqref{HJB} in $(\ux,+\infty)$.
\end{proposition}

We next state a comparison principle for viscosity sub- and supersolutions of \eqref{HJB}, in the late resolution case. It follows directly from \cite[Proposition 3.3]{achdou2026continuous}.
\begin{proposition}\label{comparison}
Assume $\theta \geq 1$, $\mathsf{u}=(\mathsf{u}_1,\mathsf{u}_2)$ and $\mathsf{v}=(\mathsf{v}_1,\mathsf{v}_2)$ are bounded viscosity sub- and supersolution of system \eqref{HJB}.We extend $\mathbf{v}_j$ at $\ux$ by setting 
$
\mathsf{v}_j(\ux)=\lim_{\substack{z\rightarrow \ux,\, z>\ux}} \mathsf{v}_j(z) .
$
Then $\mathsf{u}\leq \mathsf{v}$ in $[\ux,+\infty)$, i.e. $\mathsf{u}_j\le \mathsf{v}_j$ in $[\ux,+\infty)$ for $j=1,2$.
\end{proposition}
\cref{comparison} yields the uniqueness of the viscosity solution when $\theta \geq 1$. Reference \cite{achdou2026continuous} contains a complete study of continuous-time HA models in the case $\theta\geq 1$. It deals in particular with the existence, uniqueness and regularity of solution to the HJB equation. The following proposition summarizes results from \cite[Proposition 3.7, Proposition 3.9, Proposition 3.10, Proposition 3.15, Proposition 3.16]{achdou2026continuous} .

\begin{proposition}\label{Prop: v sol}
Assume $\theta \geq 1$. There exists a unique viscosity solution $(v_1,v_2)$ which is $C^1$ and strictly concave. Moreover, $v_j\in W^{2,\infty}_{loc}(\ux,+\infty)$ and $v_2\geq v_1$. 
\end{proposition}
Next we give some results concerning the saving policies $s_j$. The following proposition summarizes results from \cite[Proposition 3.17, Corollary 3.18, Proposition 3.21]{achdou2026continuous}.
\begin{proposition}\label{Prop: saving}
Assume $\theta \geq 1$. The optimal saving policy $s_1$ has the following properties: $s_1(x)< 0$ for all $x> \ux$ and $s_1(\ux)=0$. If furthermore
\beq\label{rho-r2}
\begin{aligned}
\lc \rho-r\rc(r\ux+y_2)^{-\gamma}+\lambda_2\lc (r\ux+y_2)^{-\gamma}-(r\ux+y_1)^{-\gamma}\rc< 0,
\end{aligned}
\eeq
then $s_2(\ux)>0$. Moreover, $Dv_1(\ux)>Dv_2(\ux)$. 
\end{proposition}

\begin{remark}
In the analysis of the numerical scheme for $\theta\geq 1$, we will focus on situations in which $s_2(\underline x)>0$,  because at the mean field  equilibrium described by (1.2),   the state constraint is not binding for agents with high income. Nevertheless, note that the methods described below apply even if \cref{rho-r2} is not satisfied. It has been shown in \cite{achdou2026continuous} that in this case $s_2(\underline x)=0$ and $s_2<0$ in some interval $(\underline x, \underline x+ \epsilon)$. If furthermore,  $\lc \rho/\theta-r\rc(r\ux+y_2)^{-1/\psi}+\lambda_2\lc (r\ux+y_2)^{-1/\psi}-(r\ux+y_1)^{-1/\psi}\rc\geq 0$, then $s_2(x)<0$ for all $x>\ux$. We refer to \Cref{Sec: Numerical} below for an example.
\end{remark}

The next result deals with the behavior of $s_2$ as $x\to +\infty$. It justifies solving the HJB equation numerically on a bounded domain $[\ux,\bar{x}]$. 
\begin{proposition}[\cite{achdou2026continuous}, Proposition 3.22]\label{Prop: s2 barx}
Assume $\theta \geq 1$ and $s_2(\ux)>0$. There exists $\widehat{x}>\ux$ such that $s_2(\widehat{x})=0$ and $s_2(x)<0$ for all $x> \widehat{x}$.
\end{proposition}
\begin{remark}
\cref{Prop: s2 barx} implies that if $\bar{x}>\widehat{x}$, then imposing an artificial state constraint at $\bar{x}$ while solving \cref{HJB} does not change the solution.
\end{remark}
It has been shown in \cite[Corollary 3.19]{achdou2026continuous} that, under the same assumptions as \cref{Prop: saving},  $\lim_{x\to \ux}D^2v_1(x)=-\infty$. This will require special care in proving the convergence of the finite difference scheme below. The next result, from \cite[Appendix C]{achdou2026continuous}, is about the asymptotic analysis of $s_1$, $D^2v_1$ and $Dv_1$ near $x=\ux$. 
\begin{proposition}\label{Prop: Dv1 ux}
Assume $\theta \geq 1$. Near $x=\ux$, the functions $s_1$, $D^2v_1$ and $Dv_1$ have the following behavior: 
\beq\label{Eq: s1D2v1 ux}
s_1(x)D^2v_1(x)= \varkappa+o(1),
\eeq
\beq\label{Eq: Dv1 ux}
Dv_1(x)\sim \rho\frac{((1-\g)v_1(\ux))^{\frac{\p^{-1}-\g}{1-\g}}}{ (r\ux+y_1)^{\p^{-1}}}-\sqrt{\frac{2\rho\varkappa(x-\ux)}{\p \lc r\ux+y_1\rc^{1+\p^{-1}}} ((1-\g)v_1(\ux))^{\frac{\p^{-1}-\g}{1-\g}}},
\eeq
where $\varkappa$ is given by
\beq\label{Eq: varkappa}
\varkappa:=\lc \frac{\rho}{\theta}-r\rc Dv_1(\ux)+\lambda_1(Dv_1(\ux)-Dv_2(\ux))-H_v(\ux,y_1,v_1(\ux),Dv_1(\ux))Dv_1(\ux).
\eeq
\end{proposition}

\begin{lemma}\label{Prop: D2v1 ux}
Assume $\theta \geq 1$. There exists $\eta\in (\ux,\bar{x})$ such that for all $x\in (\ux,\eta)$, 
\beq\label{Eq: -D2v1 bound}
0<-D^2v_1(x)\leq \sqrt{\frac{\rho\varkappa}{\p \lc r\ux+y_1\rc^{1+\p^{-1}}(x-\ux)}\cdot ((1-\g)v_1(\ux))^{\frac{\p^{-1}-\g}{1-\g}}}.
\eeq
and 
\beq\label{Eq: -s1 bound}
0<-s_1(x)<2\sqrt{\frac{\varkappa \p}{\rho} \lc r\ux+y_1\rc^{1+\p^{-1}}((1-\g)v_1(\ux))^{\frac{\p^{-1}-\g}{\g-1}}(x-\ux)}
\eeq
\end{lemma}

\begin{proof}
From \cite[Appendix C]{achdou2026continuous}, we know $D^2v_1(x)\sim Dq_1(x)$, where $q_1(x)$ and $Q_1(x)=q_1^2(x)$ satisfy
$$
\begin{aligned}
q_1(x)=-\sqrt{\frac{2\rho\varkappa(x-\ux)}{\p \lc r\ux+y_1\rc^{1+\p^{-1}}}\cdot ((1-\g)v_1(\ux))^{\frac{\p^{-1}-\g}{1-\g}}}+o(\sqrt{x-\ux}),\\
\frac{\rho^{-1}\p \lc r\ux+y_1\rc^{1+\p^{-1}}((1-\g)v_1(\ux))^{\frac{\g-\p^{-1}}{1-\g}}}{2}DQ_1(x)+o(DQ_1(x))=\varkappa+o(1).
\end{aligned}
$$
\cref{Eq: -D2v1 bound} follows from $Dq_1(x)=\frac{DQ_1(x)}{2q_1(x)}$. From \cref{Prop: saving} we know $-s_1(x)>0$ for all $x>\ux$, then \cref{Eq: -s1 bound} follows from \cref{Eq: -D2v1 bound} and \cref{Eq: s1D2v1 ux}. 
\end{proof}
\cref{Prop: Dv1 ux}, \cref{Prop: D2v1 ux} and the fact that $v_j\in W^{2,\infty}_{loc}(\ux,\bar{x}]$ will be used in \cref{Subsec: rate} below to obtain an explicit convergence rate.\par
The analysis of \cref{MFG} $(i)$  is therefore rather complete in the case $\theta\geq 1$. On the contrary, many things remain to be done in the case $0<\theta<1$. In the present paper, we restrict ourselves to supplying some results that are useful in the study of the numerical scheme.\par
Let us introduce the following system, for a given pair of functions $\left(\widetilde{v}_1,\widetilde{v}_2\right)$: 
\beq\label{Eq: HJB fixed}
\left\{
      \begin{aligned}
      \quad &\frac{\rho}{\theta} v_1(x)=H(x,y_1,\widetilde{v}_1,Dv_1)+\lambda_1(v_{2}(x)-v_1(x)),\\
&\frac{\rho}{\theta} v_2(x)=H(x,y_2,\widetilde{v}_2,Dv_2)+\lambda_2(v_{1}(x)-v_2(x)).
     \end{aligned}
   \right.
\eeq
\begin{proposition}\label{Prop: comparison fixed}
Assume $\widetilde{v}_j$ is locally Lipschitz and $\check{\su}_j\leq \widetilde{v}_j\leq \check{\sv}_j$. Let $\mathsf{u}=(\mathsf{u}_1,\mathsf{u}_2)$ and $\mathsf{v}=(\mathsf{v}_1,\mathsf{v}_2)$ be bounded viscosity sub- and supersolution of system \cref{Eq: HJB fixed}. Then $\mathsf{u}\leq \mathsf{v}$ in $[\ux,+\infty)$.
\end{proposition}
\begin{proposition}\label{Prop: sol vk iteration}
Assume $\widetilde{v}_j$ is locally Lipschitz and $\check{\su}_j\leq \widetilde{v}_j\leq \check{\sv}_j$. There exists a unique viscosity solution $(v_1,v_2)$ to \cref{Eq: HJB fixed}. Moreover, the solution $v_j\in C^1[\ux,+\infty)$. 
\end{proposition}
We define the map $\Gamma$: $(v_1,v_2)=\Gamma(\widetilde{v}_1,\widetilde{v}_2)$ if and only if $(v_1,v_2)$ is the unique viscosity solution of \cref{Eq: HJB fixed}.
A solution to \cref{HJB} can then be defined as a fixed point of $\Gamma$.\par
The proposition that follows states that the map $\Gamma$ is monotone when $0<\theta<1$:
\begin{proposition}\label{Prop: ordering}
Assume $0<\theta<1$. Let $(v_1,v_2)=\Gamma(\widetilde{v}_1,\widetilde{v}_2)$ and $(u_1,u_2)=\Gamma(\widetilde{u}_1,\widetilde{u}_2)$. If $\widetilde{u}_j\leq \widetilde{v}_j$, then $u_j\leq v_j$. 
\end{proposition}
\begin{proof}
Since $\widetilde{u}_j\leq \widetilde{v}_j$, we deduce  
$
H(x,y_j,\widetilde{u}_j,Du_j)\leq H(x,y_j,\widetilde{v}_j,Du_j)
$
from \cref{Eq:Hv>0}. This implies $(u_1,u_2)$ is a subsolution of the system of HJB equations \cref{Eq: HJB fixed} satisfied by $(v_1,v_2)$. The result then follows from \cref{Prop: comparison fixed}.
\end{proof}
\cref{Prop: ordering} yields the following invariance principle. 
\begin{proposition}\label{Prop: invariance}
We use the same notation as in \cref{Prop: ordering}. If $0<\theta<1$ and $\check{\su}_j\leq \widetilde{v}_j\leq \check{\sv}_j$, then $\check{\su}_j\leq v_j\leq \check{\sv}_j$. 
\end{proposition}
\begin{proposition}\label{Prop: vis sol bound}
For $0<\theta<1$, there exists a viscosity solution $(v_1,v_2)$ of \cref{HJB} that satisfies
$
\check{\su}_j\leq v_j\leq \check{\sv}_j.
$
\end{proposition}
\begin{proof}
Consider the iterative sequence $\lc v^{(k+1)}_1,v^{(k+1)}_2\rc=\Gamma \lc {v}^{(k)}_1,{v}^{(k)}_2\rc$, with ${v}^{(0)}_j=\check{\su}_j$. From \cref{Prop: sol vk iteration}, \cref{Prop: ordering} and \cref{Prop: invariance}, we can deduce that $v^{(k)}_j\in C^1[\ux,+\infty)$ and $\check{\su}_j\leq v^{(k)}_j\leq v^{(k+1)}_j\leq \check{\sv}_j$. Hence, $v_j^{(k)}$ converges pointwise to a function $v^\star_j$ satisfying the same bounds. Moreover, coercivity of the Hamiltonian implies $Dv^{(k)}_j$ is bounded above uniformly in $k$, hence the pointwise limit $v^\star_j$ is Lipschitz continuous. By Dini's theorem, the monotone convergence $v_j^{(k)}\to v^\star_j$ is uniform on every compact subset of $[\ux,+\infty)$. From the stability of viscosity solutions, we conclude that $v^\star=(v^\star_1,v^\star_2)$ is a constrained viscosity solution of \cref{HJB}.
\end{proof}
Observe that \cref{comparison} may not hold if $0<\theta<1$ (hence we do not claim uniqueness of solution in \cref{Prop: vis sol bound}), while the invariance principle in \cref{Prop: invariance} may not hold if $\theta \geq 1$. This is a reason why different algorithms will be needed in the two intertemporal preference cases. 
\section{The finite difference method}
\subsection{The numerical schemes}
Given $\ux\leq 0<\bar{x}$, $I\in \mathbb{N}$ and a step size $\D x=(\bar{x}-\ux)/I$, let us define the grid $\mathcal{G}^{\D x}=\{x_i:\,i\in \mathbb{N},\,i=0,\cdots,I,\,x_i=\ux+i\D x\}$. For two grid functions $U_j$ and $V_j$ defined on $\mathcal{G}^{\D x}$, we use $U_j\preceq V_j$ (resp. $U_j \succeq V_j$) to denote the ordering $U_{i,j}\leq V_{i,j}$ (resp. $U_{i,j}\geq V_{i,j}$) for all $i$. If $U=(U_1,U_2)$ and $V=(V_1,V_2)$, then $U\preceq V$ (resp. $U\succeq V$) means $U_{i,j}\leq V_{i,j}$ (resp. $U_{i,j}\geq V_{i,j}$) for all $i,j$. \par
For a Hamiltonian strictly convex w.r.t. the $p-$variable, we set
\beq\label{Eq: Hmin}
\begin{aligned}
{H}_{\min}\lc x_i,y_j,V_{i,j}\rc=\min_p H\lc x_i,y_j,V_{i,j},p\rc,\\
 {\text{and}}\quad {p}_{\min}\lc x_i,y_j,V_{i,j}\rc=\mathop{\arg\min}_{p}H\lc x_i,y_j,V_{i,j},p\rc.
\end{aligned}
\eeq
The discrete Hamiltonian is defined by
\beq\label{Eq: H num}
\begin{aligned}
&\boldsymbol{H}\lc x_i,y_j,V_{i,j},\boldsymbol{p}^{F},\boldsymbol{p}^{B}\rc\\
={}&\left\{
      \begin{aligned}
\,\, &{H}^{\uparrow}\lc x_i,y_j,V_{i,j},\boldsymbol{p}^{F}\rc\,  && \text{if}\,\,x_i=\ux,\\
&{H}^{\uparrow}\lc x_i,y_j,V_{i,j},\boldsymbol{p}^{F}\rc+{H}^{\downarrow} \lc x_i,y_j,V_{i,j},\boldsymbol{p}^{B}\rc-{H}_{\min}\lc x_i,y_j,V_{i,j}\rc\,  && \text{if}\,\,\ux<x_i<\bar{x},\\
&{H}^{\downarrow}\lc x_i,y_j,V_{i,j},\boldsymbol{p}^{B}\rc\,  && \text{if}\, \,x_i=\bar{x}.
    \end{aligned}
   \right.
       \end{aligned}
\eeq
where
\begin{equation}\label{Eq: H up}
  {H}^{\uparrow}\lc x_i,y_j,V_{i,j},\boldsymbol{p}^{F}\rc=\left\{
      \begin{aligned}
     \quad &H\lc x_i,y_j,V_{i,j},\boldsymbol{p}^{F}\rc  \quad && \text{if }\quad \boldsymbol{p}^{F}\geq {p}_{\min}\lc x_i,y_j,V_{i,j}\rc,\\
      &{H}_{\min}\lc x_i,y_j,V_{i,j}\rc\quad  && \text{if }\quad \boldsymbol{p}^{F}<{p}_{\min}\lc x_i,y_j,V_{i,j}\rc,
     \end{aligned}
   \right.
\end{equation}
and 
\begin{equation}\label{Eq: H down}
  {H}^{\downarrow} \lc x_i,y_j,V_{i,j},\boldsymbol{p}^{B}\rc=\left\{
      \begin{aligned}
      \quad &H\lc x_i,y_j,V_{i,j},\boldsymbol{p}^{B}\rc  \quad && \text{if }\quad 0\leq \boldsymbol{p}^{B}\leq {p}_{\min}\lc x_i,y_j,V_{i,j}\rc,\\
      &{H}_{\min}\lc x_i,y_j,V_{i,j}\rc \quad  && \text{if }\quad \boldsymbol{p}^{B}> {p}_{\min}\lc x_i,y_j,V_{i,j}\rc,\,\,\\
      &+\infty \quad  && \text{if }\quad \boldsymbol{p}^{B}<0.
     \end{aligned}
   \right.
\end{equation}
\begin{lemma}\label{Lemma: H monotone}
The numerical Hamiltonian $\boldsymbol{H}\lc x_i,y_j,V_{i,j},\boldsymbol{p}^{F},\boldsymbol{p}^{B}\rc$ is increasing in $\boldsymbol{p}^{F}$ and decreasing in $\boldsymbol{p}^{B}$. Moreover,
\beq\label{Eq: Hdown leq 0}
{H}^{\downarrow} \lc x_i,y_j,V_{i,j},\boldsymbol{p}^{B}\rc\leq 0\quad \text{if }\quad\boldsymbol{p}^{B}\geq 0.
\eeq
\end{lemma}
\begin{proof}
From the strict convexity of $H\lc x_i,y_j,V_{i,j},p\rc$ w.r.t. $p$ and \cref{Eq: Hmin}, we deduce for each $(x_i,y_j,V_{i,j})$, $H\lc x_i,y_j,V_{i,j},\cdot\rc$ is increasing in the domain $[{p}_{\min},+\infty)$ and decreasing in $[0,{p}_{\min})$. 
\end{proof}
In the design of Howard type algorithms, it will be crucial to take advantage of the following relationship between the discrete Hamiltonian and the utility function \cref{Eq: H from F}:
\beq\label{Eq: sF sB}
s^{F}_{i,j}=rx_i+y_j-c^{F}_{i,j},\quad s^{B}_{i,j}=rx_i+y_j-c^{B}_{i,j},
\eeq
\beq\label{Eq: H F in p}
      \begin{aligned}
{H}^{\uparrow}\lc x_i,y_j,V_{i,j},\boldsymbol{p}^{F}\rc&=\sup_{c^F_{i,j}\in [0,rx_i+y_j]}\left\{s^F_{i,j}\boldsymbol{p}^{F}+\mathcal{F}\lc c^F_{i,j},V_{i,j}\rc\right\},\\
{H}^{\downarrow} \lc x_i,y_j,V_{i,j},\boldsymbol{p}^{B}\rc&=\sup_{c^B_{i,j}\geq rx_i+y_j}\left\{s^B_{i,j}\boldsymbol{p}^{B}+\mathcal{F}\lc c^B_{i,j},V_{i,j}\rc\right\}.
\end{aligned}
\eeq
Next, we denote the consumption associated to null savings by:
\beq\label{Eq: def cbar}
\bar{c}_{j}(x_i)=rx_i+y_j.
\eeq
A straightforward computation yields
$$
{H}_{\min}\lc x_i,y_j,V_{i,j}\rc =\mathcal{F}\lc \bar{c}_{j}(x_i),V_{i,j}\rc.
$$\par
We denote the forward and backward finite difference operators by
\beq\label{Eq: FD}
\D^+{V_{i,j}}=\frac{V_{i+1,j}-V_{i,j}}{\D x},\,\,\forall \,\,0\leq i<I\,\,{\text{and}}\,\, \D^-{V_{i,j}}=\frac{V_{i,j}-V_{i-1,j}}{\D x},\,\,\forall \,\,0< i\leq I.
\eeq
For a function $v$ defined on $[\ux,+\infty)$, we denote for $x\in \mathcal{G}^{\D x}$,
\beq\label{Eq: FD v}
\begin{aligned}
&\D^+{v(x)}=\frac{v(x+\D x)-v(x)}{\D x},\,\,{\rm{if}}\,\,x<\bar{x},\\
\quad &\D^-{v(x)}=\frac{v(x)-v(x-\D x)}{\D x},\,\,{\rm{if}}\,\,\ux< x\leq \bar{x}.
 \end{aligned}
\eeq
We can set $\D^+{V_{I,j}}$, $\D^-{V_{0,j}}$, $\D^+{v(\bar{x})}$, $\D^-{v(\ux)}$ to arbitrary constants since their values have no importance in what follows. \par
The discrete version of the HJB equation \cref{HJB} is, for $i=0,\cdots,I$,
\beq\label{Eq: D-HJB H}
\left\{
      \begin{aligned}
\quad &\frac{\rho}{\theta} V_{i,1}
=\boldsymbol{H}\lc x_i,y_1,V_{i,1},\D^+{V_{i,1}},\D^-{V_{i,1}}\rc
+\lambda_1\lc V_{i,2}-V_{i,1} \rc,\\
&\frac{\rho}{\theta} V_{i,2}
=\boldsymbol{H}\lc x_i,y_2,V_{i,2},\D^+{V_{i,2}},\D^-{V_{i,2}}\rc
+\lambda_2\lc V_{i,1}-V_{i,2} \rc.
     \end{aligned}
   \right.
\eeq
Observe that by design of $\boldsymbol{H}$ in \cref{Eq: H num} at $x_i=\ux$ and $x_i=\bar{x}$, the equation \cref{Eq: D-HJB H} at $i=0$ and $i=I$ have taken into account the state constraints because $H^\uparrow$  (resp. $H^\downarrow$) corresponds to nonnegative (nonpositive) savings. If $\D^-{V_{i,j}}>0$ for all $1\leq i\leq I$ (this will be proved in \cref{Prop: positive Dv} below), then from the first order condition of optimality, we deduce the numerical optimal consumption 
\beq\label{Eq: CF CB}
\left\{
\begin{aligned}
\quad &c^{F,*}_{i,j}=\min\left\{\rho^{\p}\lc \D^+{V_{i,j}}\rc^{-\p}\lc(1-\g)V_{i,j}\rc^{\frac{1-\g \p}{1-\g}},\bar{c}_{j}(x_i)\right\}\quad \text{for}\,\,x_i<\bar{x},\\
&c^{B,*}_{i,j}=\max\left\{\rho^{\p}\lc \D^-{V_{i,j}}\rc^{-\p}\lc(1-\g)V_{i,j}\rc^{\frac{1-\g \p}{1-\g}},\bar{c}_{j}(x_i)\right\}\quad \text{for}\,\,x_i>\ux,\\
&c^{B,*}_{0,j}=r\ux+y_j,\quad c^{F,*}_{I,j}=r\bar{x}+y_j.
\end{aligned}
   \right.
\eeq
An equivalent formulation of \cref{Eq: D-HJB H} is then:
\beq\label{Eq: D-HJB F}
\left\{
      \begin{aligned}
\quad \frac{\rho}{\theta} V_{i,1}
={}&s^{F,*}_{i,1}\D^+{V_{i,1}}+s^{B,*}_{i,1}\D^-{V_{i,1}}+\mathcal{F}\lc c^{F,*}_{i,1},V_{i,1}\rc+\mathcal{F}\lc c^{B,*}_{i,1},V_{i,1}\rc\\
&-\mathcal{F}\lc \bar{c}_{i,1},V_{i,1}\rc+\lambda_1\lc V_{i,2}-V_{i,1} \rc,\\
\frac{\rho}{\theta} V_{i,2}
={}& s^{F,*}_{i,2}\D^+{V_{i,2}}+s^{B,*}_{i,2}\D^-{V_{i,2}}+\mathcal{F}\lc c^{F,*}_{i,2},V_{i,2}\rc+\mathcal{F}\lc c^{B,*}_{i,2},V_{i,2}\rc\\
&-\mathcal{F}\lc \bar{c}_{i,2},V_{i,2}\rc+\lambda_2\lc V_{i,1}-V_{i,2} \rc,
     \end{aligned}
   \right.
\eeq
where $\mathcal{F}$ is defined as in \cref{EZ flow}, $s^{F,*}_{i,j}=rx_i+y_j-c^{F,*}_{i,j}$ and $s^{B,*}_{i,j}=rx_i+y_j-c^{B,*}_{i,j}$. By construction $s^{F,*}_{i,j}\geq 0$ and $s^{B,*}_{i,j}\leq 0$, hence the scheme is in upwind form.\par
With $\bar{c}_j$ defined in \cref{Eq: def cbar}, and given $\eps>0$, let us introduce the following regularized consumption policies: 
\beq\label{Eq: CF CB proj}
\begin{aligned}
\left\{
\begin{aligned}
\quad  &c^{F,*}_{i,j}=\min\left\{\rho^{\p}\lc \D^+{V_{i,j}}\rc_+^{-\p}\lc(1-\g)V_{i,j}\rc^{\frac{1-\g \p}{1-\g}},\bar{c}_{j}(x_i)\right\}\quad \text{for}\,\,x_i<\bar{x},\\
&c^{B,*}_{i,j}=\max\left\{\min\left\{1/\eps,\rho^{\p}\lc \D^-{V_{i,j}}\rc_+^{-\p}\lc(1-\g)V_{i,j}\rc^{\frac{1-\g \p}{1-\g}}\right\},\bar{c}_{j}(x_i)\right\}\quad \text{for}\,\,x_i>\ux,\\
&c^{B,*}_{0,j}=r\ux+y_j,\quad c^{F,*}_{I,j}=r\bar{x}+y_j.
\end{aligned}
   \right.
   \end{aligned}
\eeq
Let us introduce the regularized Hamiltonian, with $\left(c^{F,*},c^{B,*}\right)$ and $V$ satisfying \cref{Eq: CF CB proj}:
\beq\label{Eq: HR F}
\begin{aligned}
&\boldsymbol{H}_{\eps}\lc x_i,y_j,V_{i,j},\boldsymbol{p}^{F},\boldsymbol{p}^{B}\rc\\
={}&
s^{F,*}_{i,j}\boldsymbol{p}^{F}+s^{B,*}_{i,j}\boldsymbol{p}^{B}+\mathcal{F}\lc c^{F,*}_{i,j},V_{i,j}\rc+\mathcal{F}\lc c^{B,*}_{i,j},V_{i,j}\rc-\mathcal{F}\lc \bar{c}_{j}(x_i),V_{i,j}\rc.
\end{aligned}
\eeq
 We introduce a regularized version of \cref{Eq: D-HJB H}:
\beq\label{Eq: D-HJB HR}
\left\{
      \begin{aligned}
\quad &\frac{\rho}{\theta} V_{i,1}
=\boldsymbol{H}_{\eps}\lc x_i,y_1,V_{i,1},\D^+{V_{i,1}},\D^-{V_{i,1}}\rc
+\lambda_1\lc V_{i,2}-V_{i,1} \rc,\\
&\frac{\rho}{\theta} V_{i,2}
=\boldsymbol{H}_{\eps}\lc x_i,y_2,V_{i,2},\D^+{V_{i,2}},\D^-{V_{i,2}}\rc
+\lambda_2\lc V_{i,1}-V_{i,2} \rc,
     \end{aligned}
   \right.
\eeq
supposing $\lc(1-\g)V_{i,j}\rc^{\frac{1-\g \p}{1-\g}}$ is bounded (we shall see later that this property is true, uniformly with respect to $\eps$). Throughout the rest of the present paper, we assume
\beq\label{Eq: eps lower}
1/{\eps}>r\bar{x}+y_2.
\eeq
We observe that the regularized Hamiltonian in \cref{Eq: D-HJB HR} is defined even if $\D^-{V_{i,j}}<0$. We also set
\beq\label{Eq: H from F eps}
      \begin{aligned}
{H}^{\uparrow}_{\eps}\lc x_i,y_j,V_{i,j},\D^+{V_{i,j}}\rc&={H}^{\uparrow}\lc x_i,y_j,V_{i,j},\D^+{V_{i,j}}\rc,\\
{H}^{\downarrow}_{\eps} \lc x_i,y_j,V_{i,j},\D^-{V_{i,j}}\rc&=\sup_{rx_i+y_j\leq c^B_{i,j}\leq 1/\eps}\left\{s^B_{i,j}\D^-{V_{i,j}}+\mathcal{F}\lc c^B_{i,j},V_{i,j}\rc\right\}.
\end{aligned}
\eeq
This strategy of regularizing the Hamiltonian by truncations has been used in \cite{lauriere2023policy,tang2023learning}.\par
\begin{remark}
We observe that the FPK equations \cref{MFG} $(ii)$ are linear, hence their numerical approximation is standard and will not be studied in details here. We restrict ourselves to saying that, defining $(s^F,s^B)$ as in \cref{Eq: CF CB proj}, the discrete scheme for \cref{MFG} $(ii)$ is in the form: 
\beq\label{Eq: FPK scheme}
\frac{s^{F,*}_{i,j}g_{i,j}-s^{B,*}_{i,j}g_{i,j}}{\D x}+\lambda_jg_{i,j}=\frac{s^{F,*}_{i-1,j}g_{i-1,j}-s^{B,*}_{i+1,j}g_{i+1,j}}{\D x}+\lambda_{\bar \jmath}g_{i,\bar \jmath}.
\eeq
For $x_i>\ux$, $g_{i,j}$ approximates the density $g_j(x_i)$. If $m_j$ exhibits a Dirac mass at $\ux$, then its weight $\mu_j$ is approximated by $g_{0,j}\D x$. 
\end{remark}
\subsection{First elements in the analysis of the scheme}
We first give some results on the scheme which hold for all $\theta>0$. \Cref{Sec: late} and \Cref{Sec: early} will respectively contain results particular to the cases $\theta\ge 1$ and $0<\theta<1$.

\begin{lemma}\label{Lemma: num H_V}
Assume $\theta > 1$ (resp., $0<\theta<1$). The numerical Hamiltonian $\boldsymbol{H}\lc x_i,y_j,V_{i,j},\boldsymbol{p}^{F},\boldsymbol{p}^{B}\rc$ is decreasing (resp., increasing) in $V_{i,j}$: if $U_{i,j}< V_{i,j}<0$ then 
$
\boldsymbol{H}\lc x_i,y_j,V_{i,j},\boldsymbol{p}^{F},\boldsymbol{p}^{B}\rc<\,\, (resp., >)\,\, \boldsymbol{H}\lc x_i,y_j,U_{i,j},\boldsymbol{p}^{F},\boldsymbol{p}^{B}\rc.
$
\end{lemma}
\begin{proof}
We only give the proof in the case $\theta \geq 1$, the proof in the other case being very similar. Let $c^{F,*}_{i,j}$ and $c^{B,*}_{i,j}$ be the maximizers given by \cref{Eq: H F in p}.
If $\ux<x_i<\bar{x}$, then
$
\mathcal{F}\lc c^{B,*}_{i,j},V_{i,j}\rc< \mathcal{F}\lc c^{B,*}_{i,j},U_{i,j}\rc
$
follows from $U_{i,j}< V_{i,j}$ and \cref{F monotone}. From \cref{Eq: H F in p} we know $c^{F,*}_{i,j}\leq \bar{c}_{j}(x_i)$, a straightforward computation with \cref{EZ flow} leads to
$$
\mathcal{F}\lc c^{F,*}_{i,j},V_{i,j}\rc-\mathcal{F}\lc \bar{c}_{j}(x_i),V_{i,j}\rc \leq \mathcal{F}\lc c^{F,*}_{i,j},U_{i,j}\rc-\mathcal{F}\lc \bar{c}_{j}(x_i),U_{i,j}\rc.
$$
This yields 
\beq
\begin{aligned}
&\boldsymbol{H}\lc x_i,y_j,V_{i,j},\boldsymbol{p}^{F},\boldsymbol{p}^{B}\rc\\
< {}&
s^{F,*}_{i,j}\boldsymbol{p}^{F}+s^{B,*}_{i,j}\boldsymbol{p}^{B}+\mathcal{F}\lc c^{B,*}_{i,j},U_{i,j}\rc+\mathcal{F}\lc c^{F,*}_{i,j},U_{i,j}\rc-\mathcal{F}\lc \bar{c}_{j}(x_i),U_{i,j}\rc\\
\leq {}& \boldsymbol{H}\lc x_i,y_j,U_{i,j},\boldsymbol{p}^{F},\boldsymbol{p}^{B}\rc.
\end{aligned}
\eeq
At the boundaries $x_i=\ux$ or $x_i=\bar{x}$, the conclusion is obvious from \cref{Eq: H up} and \cref{Eq: H down}. 
\end{proof}
The same result holds for the regularized Hamiltonian defined in \cref{Eq: HR F}, the proof being essentially the same as that of \cref{Lemma: num H_V}.
\begin{lemma}\label{Lemma: num H_Veps}
Assume $\theta > 1$ (resp., $0<\theta<1$). The numerical Hamiltonian $\boldsymbol{H}_{\eps}\lc x_i,y_j,V_{i,j},\boldsymbol{p}^{F},\boldsymbol{p}^{B}\rc$ is decreasing (resp., increasing) in $V_{i,j}$: if $U_{i,j}< V_{i,j}<0$ then 
$
\boldsymbol{H}_{\eps}\lc x_i,y_j,V_{i,j},\boldsymbol{p}^{F},\boldsymbol{p}^{B}\rc<\,\, (resp., >)\,\, \boldsymbol{H}_{\eps}\lc x_i,y_j,U_{i,j},\boldsymbol{p}^{F},\boldsymbol{p}^{B}\rc.
$
\end{lemma}
We now introduce the discrete sub and supersolution to \cref{Eq: D-HJB H}, analogous to the functions introduced in \cref{def vis sol}. 
\begin{definition}
We say that $\mathsf{U}=(\mathsf{U}_1,\mathsf{U}_2)$ is a discrete subsolution of \cref{Eq: D-HJB H} if for all $i,j$ such that $i=0,\cdots,I$ and $j\in \{1,2\}$,
\beq\label{Ineq: sub}
\frac{\rho}{\theta} \mathsf{U}_{i,j}\leq \boldsymbol{H}\lc x_i,y_j,\mathsf{U}_{i,j},\D^+{\mathsf{U}_{i,j}},\D^-{\mathsf{U}_{i,j}}\rc+\lambda_j(\mathsf{U}_{i,\bar \jmath}-\mathsf{U}_{i,j}).
\eeq 
Respectively, $\mathsf{V}=(\mathsf{V}_1,\mathsf{V}_2)$ is a discrete supersolution of \cref{Eq: D-HJB H} if for all $i,j$ such that $i=0,\cdots,I$ and $j\in \{1,2\}$, 
\beq\label{Ineq: super}
\frac{\rho}{\theta} \mathsf{V}_{i,j}\geq \boldsymbol{H}\lc x_i,y_j,\mathsf{V}_{i,j},\D^+{\mathsf{V}_{i,j}},\D^-{\mathsf{V}_{i,j}}\rc+\lambda_j(\mathsf{V}_{i,\bar \jmath}-\mathsf{V}_{i,j}).
\eeq 
\end{definition}
In what follows, we introduce discrete barriers as in \cref{Prop:sub-super}. They will provide bounds on the numerical solutions.  
\begin{proposition}\label{Prop:sub-super disc}
With the constant $b$ defined in \cref{tab1}, we consider the grid functions
$$
\check{U}_{i,1}=\check{U}_{i,2}=\frac{(rx_i+y_1)^{1-\gamma}}{1-\gamma},\quad \check{V}_{i,1}=\check{V}_{i,2}=\frac{(b(x_i+y_2/r))^{1-\gamma}}{1-\gamma}.
$$ 
The pairs $(\check{U}_{i,1},\check{U}_{i,2})$ and $(\check{V}_{i,1},\check{V}_{i,2})$ are respectively a sub- and a supersolution of \cref{Eq: D-HJB H}.
\end{proposition}
\begin{proof}
{\it{Part 1: subsolution.}} A direct computation shows $\mathcal{F}\lc \bar{c}_{1}(x_i),\check{U}_{i,1}\rc=\frac{\rho}{\theta}\check{U}_{i,1}$. We first establish inequality \cref{Ineq: sub} at $j=1$.
Let us set $c^F_{i,1}=c^B_{i,1}=\bar{c}_1(x_i)$, hence $s^F_{i,1}=s^B_{i,1}=0$ if $0<i<I$. Similarly, we set $s^F_{0,1}=s^B_{I,1}=0$. Then,
\begin{align*}
&\boldsymbol{H}\lc x_i,y_1,\check{U}_{i,1},\D^+{\check{U}_{i,1}},\D^-{\check{U}_{i,1}}\rc\\
\geq  {}& s^F_{i,1}\D^+{\check{U}_{i,1}}+s^B_{i,1}\D^-{\check{U}_{i,1}}+\mathcal{F}\lc c^F_{i,1},\check{U}_{i,1}\rc+\mathcal{F}\lc c^B_{i,1},\check{U}_{i,1}\rc-\mathcal{F}\lc \bar{c}_{i,1},\check{U}_{i,1}\rc\\
={}&\mathcal{F}\lc \bar{c}_{1}(x_i),\check{U}_{i,1}\rc=\frac{\rho}{\theta}\check{U}_{i,1}.
\end{align*}
 Let us now turn to inequality \cref{Ineq: sub} at $j=2$. Set $c^F_{i,2}=c^B_{i,2}=\bar{c}_2(x_i)=rx_i+y_2$, hence $s^F_{i,2}=s^B_{i,2}=0$, then
$$
\boldsymbol{H}\lc x_i,y_2,\check{U}_{i,2},\D^+{\check{U}_{i,2}},\D^-{\check{U}_{i,2}}\rc\geq \mathcal{F}\lc \bar{c}_{2}(x_i),\check{U}_{i,2}\rc\geq \mathcal{F}\lc \bar{c}_{1}(x_i),\check{U}_{i,2}\rc\geq \frac{\rho}{\theta}\check{U}_{i,2}.
$$
Therefore, for both $j=1$ and $j=2$, $\lambda_j(\check{U}_{i,\bar \jmath}-\check{U}_{i,j})=0$, and 
$$
\frac{\rho}{\theta} \check{U}_{i,j}\leq \boldsymbol{H}\lc x_i,y_j,\check{U}_{i,j},\D^+{\check{U}_{i,j}},\D^-{\check{U}_{i,j}}\rc+\lambda_j(\check{U}_{i,\bar \jmath}-\check{U}_{i,j}).
$$
{\it{Part 2: supersolution.}} We aim at proving the inequality \cref{Ineq: super} at $j=2$. From $\check{V}_{i,j}=\check{\sv}_{j}(x_i)$, where $\check{\sv}_{j}$ is defined in \cref{Eq: sub super}, we know ${p}_{\min}\lc x_i,y_2,\check{\sv}_{2}(x_i)\rc={p}_{\min}\lc x_i,y_2,\check{V}_{i,2}\rc$. Straightforward computation leads to
$
\frac{\rho}{\theta}\check{\sv}_{2}(x_i)=H(x_i,y_2,\check{\sv}_{2}(x_i),D\check{\sv}_{2}(x_i)),
$
for $x_i>\ux$. Since $\check{\sv}_{2}$ is strictly concave, we can infer that
\beq\label{Eq: Dv_2}
\D^+{\check{V}_{i,2}}<D\check{\sv}_{2}(x_i)<\D^-{\check{V}_{i,2}}.
\eeq
From \cref{Lemma: H monotone} and \cref{Eq: Dv_2}, it follows
\begin{align*}
H(x_i,y_2,\check{\sv}_{2}(x_i),D\check{\sv}_{2}(x_i))
={}&\boldsymbol{H}\lc x_i,y_2,\check{V}_{i,2},D\check{\sv}_{2}(x_i),D\check{\sv}_{2}(x_i)\rc \\
\geq {}&\boldsymbol{H}\lc x_i,y_2,\check{V}_{i,2},\D^+{\check{V}_{i,2}},\D^-{\check{V}_{i,2}}\rc,
\end{align*}
and we obtain
$$
\frac{\rho}{\theta}\check{V}_{i,2}\geq \boldsymbol{H}\lc x_i,y_2,\check{V}_{i,2},\D^+{\check{V}_{i,2}},\D^-{\check{V}_{i,2}}\rc.
$$
Finally we prove the inequality \cref{Ineq: super} at $j=1$. From $y_2>y_1$ and $\check{V}_{i,1}=\check{V}_{i,2}$, we immediately obtain that 
$$
\boldsymbol{H}\lc x_i,y_2,\check{V}_{i,2},\D^+{\check{V}_{i,2}},\D^-{\check{V}_{i,2}}\rc>\boldsymbol{H}\lc x_i,y_1,\check{V}_{i,1},\D^+{\check{V}_{i,1}},\D^-{\check{V}_{i,1}}\rc,
$$
hence 
$
\frac{\rho}{\theta}\check{V}_{i,1}\geq \boldsymbol{H}\lc x_i,y_1,\check{V}_{i,1},\D^+{\check{V}_{i,1}},\D^-{\check{V}_{i,1}}\rc.
$
\end{proof}
 Since for any $U_{i,j}$,
$
\boldsymbol{H}_{\eps}\lc x_i,y_j,U_{i,j},\D^+{U_{i,j}},\D^-{U_{i,j}}\rc \leq \boldsymbol{H}\lc x_i,y_j,U_{i,j},\D^+{U_{i,j}},\D^-{U_{i,j}}\rc,
$
we have the following results.
\begin{lemma}\label{Prop: sub super HR1}
A subsolution of \cref{Eq: D-HJB HR} is also a subsolution of \cref{Eq: D-HJB H}. A supersolution of \cref{Eq: D-HJB H} is a supersolution of \cref{Eq: D-HJB HR}.
\end{lemma}
\begin{proposition}\label{Prop: sub super HR2}
The grid functions $\check{U}$ and $\check{V}$, defined in \cref{Prop:sub-super disc}, are respectively a constrained sub- and supersolution of \cref{Eq: D-HJB HR}.  
\end{proposition}
\begin{proof}
The fact that $\check{V}$ is supersolution to \cref{Eq: D-HJB HR} follows directly from \cref{Prop: sub super HR1}. To show that $\check{U}$ is subsolution to \cref{Eq: D-HJB HR}, we only need to observe that $c^F_{i,1}=c^B_{i,1}=\bar{c}_1(x_i)$ is still an admissible control, then we follow the same proof as that of \cref{Prop:sub-super disc}.
\end{proof}

\section{The case \texorpdfstring{$\theta \geq 1$}{theta >= 1}: existence, iterative algorithm and error estimate}\label{Sec: late}
Throughout this section we make the standing assumption $\theta \geq 1$.
\subsection{Existence, uniqueness and gradient estimates}
The comparison principle that follows is the discrete counterpart to \cref{comparison}. 
\begin{proposition}\label{Prop: comparison}
If ${U}$ and ${V}$ are respectively a sub and supersolution of \cref{Eq: D-HJB H}, then ${U}\preceq  {V}$.
\end{proposition}
\begin{proof} Suppose 
$
{U}_{i^*,j^*}-{V}_{i^*,j^*}=\max_{i,j}\{{U}_{i,j}-{V}_{i,j}\}=\delta>0.
$
Since ${U}$ is a subsolution of \cref{Eq: D-HJB H},
\beq
\frac{\rho}{\theta} {U}_{i^*,j^*}\leq \boldsymbol{H}\lc x_{i^*},y_{j^*},{U}_{i^*,j^*},\D^+{{U}_{i^*,j^*}},\D^-{{U}_{i^*,j^*}}\rc+\lambda_{j^*}({U}_{i^*,\bar \jmath^*}-{U}_{i^*,j^*}).
\eeq
Observe that $\D^+{{V}_{i^*,j^*}}\geq \D^+{{U}_{i^*,j^*}}$ if $0\leq i^*\leq I-1$, and $\D^-{{V}_{i^*,j^*}}\leq \D^-{{U}_{i^*,j^*}}$ if $1\leq i^*\leq I$. \cref{Lemma: H monotone} and \cref{Eq: H num} then yield
$$
\begin{aligned}
&\boldsymbol{H}\lc x_{i^*},y_{j^*},{U}_{i^*,j^*},\D^+{{U}_{i^*,j^*}},\D^-{{U}_{i^*,j^*}}\rc\\
\leq {}&\boldsymbol{H}\lc x_{i^*},y_{j^*},{U}_{i^*,j^*},\D^+{{V}_{i^*,j^*}},\D^-{{V}_{i^*,j^*}}\rc.
\end{aligned}
$$
Since $\theta \geq 1$ and ${U}_{i^*,j^*}>{V}_{i^*,j^*}$, \cref{Lemma: num H_V} leads to 
\beq\label{Eq: HU HV}
\begin{aligned}
&\boldsymbol{H}\lc x_{i^*},y_{j^*},{U}_{i^*,j^*},\D^+{{V}_{i^*,j^*}},\D^-{{V}_{i^*,j^*}}\rc\\
\leq {}&\boldsymbol{H}\lc x_{i^*},y_{j^*},{V}_{i^*,j^*},\D^+{{V}_{i^*,j^*}},\D^-{{V}_{i^*,j^*}}\rc.
\end{aligned}
\eeq
This implies, together with ${U}_{i^*,j^*}-{V}_{i^*,j^*}\geq {U}_{i^*,\bar \jmath^*}-{V}_{i^*,\bar \jmath^*}$, that
\beq
\frac{\rho}{\theta} {U}_{i^*,j^*}\leq \boldsymbol{H}\lc x_{i^*},y_{j^*},{V}_{i^*,j^*},\D^+{{V}_{i^*,j^*}},\D^-{{V}_{i^*,j^*}}\rc+\lambda_{j^*}({V}_{i^*,\bar \jmath^*}-{V}_{i^*,j^*}).
\eeq 
With the supersolution inequality for ${V}$, we get $\rho \delta/\theta\leq 0$, a contradiction. 
\end{proof}
The following barrier property is then deduced from \cref{Prop: comparison} and \cref{Prop:sub-super disc}. 
\begin{proposition}\label{Prop: sol bound late}
Let $(V_1,V_2)$ be a solution of \cref{Eq: D-HJB H}. Then 
\beq\label{Eq: sol bound}
\check{U}\preceq  V\preceq  \check{V}.
\eeq
\end{proposition}

We now introduce a discrete Perron's method, which gives the existence of a solution $V$ to \cref{Eq: D-HJB H}.
\begin{proposition}\label{Prop: Perron}
Suppose that for all $i,j$,
\beq\label{Eq: Perron}
U_{i,j}:=\sup \left\{Z_{i,j}:\, (Z_{1},Z_{2})\,{\text{is a subsolution of \cref{Eq: D-HJB H} nondecreasing with respect to}}\,\, i\right\}
\eeq
Then $U=(U_1,U_2)$ is a solution of \cref{Eq: D-HJB H}. 
\end{proposition}
\begin{proof}
We first observe that $U_j$ is nondecreasing with respect to $i$. \par
{Let us first prove $U$ is a subsolution.} Given the grid node $x_i$ and a positive number $\eps$, \cref{Eq: Perron} implies that there exists a subsolution $(Z_{1},Z_{2})$ such that $Z_{i,j}>{U}_{i,j}-\eps \D x$. Moreover, from the maximality of $U$ we infer 
$
U_{i+1,j}\geq Z_{i+1,j}$, $\quad U_{i-1,j}\geq Z_{i-1,j}$, $\quad U_{i,\bar \jmath}\geq Z_{i,\bar \jmath}$.
 This implies $\D^+{U_{i,j}}\geq \D^+{Z_{i,j}}-\eps$ and $\D^-{U_{i,j}}\leq \D^-{Z_{i,j}}+\eps$. From \cref{Lemma: H monotone}, we deduce 
 $$
 \boldsymbol{H}\lc x_i,y_j,{U}_{i,j},\D^+{{U}_{i,j}},\D^-{{U}_{i,j}}\rc \geq \boldsymbol{H}\lc x_i,y_j,Z_{i,j}+\eps \D x,\D^+{Z_{i,j}}-\eps,\D^-{Z_{i,j}}+\eps\rc.
 $$
From the local Lipschitz continuity of $\boldsymbol{H}$ and letting $\eps\to 0$, we deduce that $U_j$ satisfies the subsolution inequality at $x_i$. \par
Next suppose $U$ is not a solution to \cref{Eq: D-HJB H}, i.e. there exists $i^*,{j^*}$ and $\delta>0$ such that
\beq\label{Eq: sub U}
\left\{
      \begin{aligned}
 &\frac{\rho}{\theta} U_{i^*,j^*}
-\boldsymbol{H}\lc x_{i^*},y_{j^*},U_{i^*,j^*},\D^+{U_{i^*,j^*}},\D^-{U_{i^*,j^*}}\rc
-\lambda_{j^*}\lc U_{i^*,\bar \jmath^*}-U_{i^*,j^*} \rc\leq -\delta,\\
&\frac{\rho}{\theta} U_{i^*,\bar \jmath^*}
-\boldsymbol{H}\lc x_{i^*},y_{\bar \jmath^*},U_{i^*,\bar \jmath^*},\D^+{U_{i^*,\bar \jmath^*}},\D^-{U_{i^*,\bar \jmath^*}}\rc
-\lambda_{\bar \jmath^*}\lc U_{i^*,j^*}-U_{i^*,\bar \jmath^*} \rc\leq 0.
     \end{aligned}
   \right.
\eeq
We then make out two cases.\par
\fbox{\it{Case 1.}} $\D^+{U_{i^*,j^*}}>0$. It is possible to choose a constant $\eta>0$ and a grid function $W_{i,j^*}$ such that
\beq\label{Eq: def W}
W_{i,j}=U_{i,j} \,\,\forall i\neq i^*,\,\,W_{i^*,j^*}=U_{i^*,j^*}+\eta \D x\,\, {\text{and}}\,\,W_{i^*,j^*}\leq U_{i^*+1,j^*}.
\eeq
Plugging this information into the second line of \cref{Eq: sub U} gives
\beq\label{Eq: sub U Perron}
\begin{aligned}
&\frac{\rho}{\theta} U_{i^*,\bar \jmath^*}
-\boldsymbol{H}\lc x_{i^*},y_{\bar \jmath^*},U_{i^*,\bar \jmath^*},\D^+{U_{i^*,\bar \jmath^*}},\D^-{U_{i^*,\bar \jmath^*}}\rc
-\lambda_{\bar \jmath^*}\lc W_{i^*,j^*}-U_{i^*,\bar \jmath^*} \rc\\
\leq{}& -\lambda_{\bar \jmath^*}\eta \D x<0. 
\end{aligned}
\eeq
It is always possible to choose $\eta$ sufficiently small such that $\lc\frac{\rho}{\theta}+\lambda_{j^*}\rc \eta \D x\leq {\delta}/{4}$, hence
\beq\label{Eq: W-U}
\frac{\rho}{\theta} W_{i^*,j^*}
-\lambda_{j^*}\lc U_{i^*,\bar \jmath^*}-W_{i^*,j^*} \rc\leq \frac{\rho}{\theta} U_{i^*,j^*}
-\lambda_{j^*}\lc U_{i^*,\bar \jmath^*}-U_{i^*,j^*}\rc+\frac{\delta}{4}.
\eeq
Since $\D^+{W_{i^*,j^*}}=\D^+{U_{i^*,j^*}}-\eta$,  $\D^-{W_{i^*,j^*}}=\D^-{U_{i^*,j^*}}+\eta$, we obtain for sufficiently small $\eta$ that 
\beq\label{Eq: -H Perron}
\begin{aligned}
&-\boldsymbol{H}\lc x_{i^*},y_{j^*},W_{i^*,j^*},\D^+{W_{i^*,j^*}},\D^-{W_{i^*,j^*}}\rc\\
\leq {}&-\boldsymbol{H}\lc x_{i^*},y_{j^*},U_{i^*,j^*},\D^+{U_{i^*,j^*}},\D^-{U_{i^*,j^*}}\rc+\frac{\delta}{4}.
\end{aligned}
\eeq
Summing up \cref{Eq: W-U} and \cref{Eq: -H Perron}, we obtain from the first equation in \cref{Eq: sub U} that 
\beq\label{Eq: sub W Perron}
\frac{\rho}{\theta} W_{i^*,j^*}-\boldsymbol{H}\lc x_{i^*},y_{j^*},W_{i^*,j^*},\D^+{W_{i^*,j^*}},\D^-{W_{i^*,j^*}}\rc-\lambda_{j^*}\lc U_{i^*,\bar \jmath^*}-W_{i^*,j^*} \rc\leq -\frac{\delta}{2}.
\eeq
From \cref{Eq: sub U Perron} and \cref{Eq: sub W Perron}, $(W_j,U_{\bar \jmath^*})$ is a subsolution while $W_{i^*,j^*}>U_{i^*,j^*}$. This contradicts the definition of $U$, see \cref{Eq: Perron}. \par
\fbox{\it{Case 2.}} $\D^+{U_{i^*,j^*}}=0$. We need to make out several sub-cases. \\
\fbox{\it{Case 2.1.}} Suppose $\D^+{U_{i^*,j^*}}=0$, $i^*=I-1$. Then, 
$
{H}^{\downarrow} \lc x_{I},y_{j^*},U_{I,j^*},\D^-{U_{I,j^*}}\rc=0.
$
We derive from \cref{Eq: sub U} that
$$
\frac{\rho}{\theta} U_{I,j^*}-\lambda_{j^*}\lc U_{I,\bar \jmath^*}-U_{I,j^*} \rc\leq -\delta
$$
Defining $W_{j^*}$ as in \cref{Eq: def W}, we obtain $\D^-{W_{I,j^*}}=\eta$. We can choose $\eta$ sufficiently small such that \cref{Eq: W-U} holds at $i^*=I$ and 
$
-{H}^{\downarrow} \lc x_{I},y_{j^*},W_{I,j^*},\eta\rc\leq \frac{\delta}{4}.
$
This yields
$$
\frac{\rho}{\theta} W_{I,j^*}-\lambda_{j^*}\lc U_{i^*,\bar \jmath^*}-W_{i^*,j^*} \rc-{H}^{\downarrow} \lc x_{I},y_{j^*},W_{I,j^*},\eta\rc\leq -\frac{\delta}{2},
$$
i.e. the same contradiction as in {\it{Case 1.}}
 \\
\fbox{\it{Case 2.2.}} Suppose $\D^+{U_{i^*,j^*}}=0$, $i^*<I-1$ and there exists $l\leq I-1-i^*$, such that $\D^+{U_{i^*+l,j^*}}>0$ and 
\beq\label{Eq: Ui=Ui+l}
U_{i^*,j^*}=\cdots=U_{i^*+l,j^*}.
\eeq
 Since $U_{\bar \jmath^*}$ is nondecreasing w.r.t. $i$, $U_{i^*+l,\bar \jmath^*}\geq U_{i^*,\bar \jmath^*}$. With \cref{Eq: Ui=Ui+l}, we deduce
\beq\label{Eq: lambda i+l}
 -\lambda_{j^*}\lc U_{i^*+l,\bar \jmath^*}-U_{i^*+l,j^*} \rc\leq -\lambda_{j^*}\lc U_{i^*,\bar \jmath^*}-U_{i^*,j^*} \rc.
\eeq
Since $\D^+{U_{i^*+l,j^*}}>\D^+{U_{i^*,j^*}}$ and $\D^-{U_{i^*+l,j^*}}=0 \leq \D^-{U_{i^*,j^*}}$. \cref{Lemma: H monotone} then yields
\beq\label{Eq: H i+l}
\begin{aligned}
&\boldsymbol{H}\lc x_{i^*+l},y_{j^*},U_{i^*+l,j^*},\D^+{U_{i^*+l,j^*}},\D^-{U_{i^*+l,j^*}}\rc\\
\geq{}& \boldsymbol{H}\lc x_{i^*},y_{j^*},U_{i^*+l,j^*},\D^+{U_{i^*+l,j^*}},\D^-{U_{i^*+l,j^*}}\rc\\
\geq {}&\boldsymbol{H}\lc x_{i^*},y_{j^*},U_{i^*,j^*},\D^+{U_{i^*,j^*}},\D^-{U_{i^*,j^*}}\rc.
\end{aligned}
\eeq
Combining \cref{Eq: Ui=Ui+l}, \cref{Eq: lambda i+l} and \cref{Eq: H i+l}, we obtain 
$$
\begin{aligned}
&\frac{\rho}{\theta} U_{i^*+l,j^*}-\boldsymbol{H}\lc x_{i^*+l},y_{j^*},U_{i^*+l,j^*},\D^+{U_{i^*+l,j^*}},\D^-{U_{i^*+l,j^*}}\rc\\
&-\lambda_{j^*}\lc U_{i^*+l,\bar \jmath^*}-U_{i^*+l,j^*} \rc
\leq -\delta.
\end{aligned}
$$
We can then consider the perturbation
$$
W_{i,j}=U_{i,j} \,\,\forall i\neq i^*+l,\,\,W_{i^*+l,j^*}=U_{i^*+l,j^*}+\eta \D x\,\, {\text{and}}\,\,W_{i^*+l,j^*}\leq U_{i^*+l+1,j^*}.
$$
Similarly to {\it{Case 1}}, we obtain $(W_{j^*},U_{\bar \jmath^*})$ is a subsolution while $W_{i^*+l,j^*}>U_{i^*+l,j^*}$. \\
\fbox{\it{Case 2.3.}} Suppose $\D^+{U_{i^*,j^*}}=0$, $i^*<I-1$ and the $l$ defined in {\it{Case 2.2.}} does not exist, i.e. $U_{i^*,j^*}=\cdots=U_{I,j^*}$. Then $\D^+{U_{I-1,j^*}}=0$, $(U_{j^*},U_{\bar \jmath^*})$ satisfies \cref{Eq: sub U} at $I$ and the desired result follows from {\it{Case 2.1}}. 
\end{proof}

Next, we prove that the grid functions ${V_{i,j}}$ are strictly increasing w.r.t. $i$.  
\begin{proposition}\label{Prop: positive Dv}
Let $V$ be a solution to \cref{Eq: D-HJB H}. The finite difference $\D^-{V_{i,j}}$ is strictly positive for all $i\geq1$ and $j\in \{1,2\}$.
\end{proposition}
\begin{proof}
From the coercivity of the discrete Hamiltonian, $\D^-{V_{i,j}}\geq 0$. We argue by contradiction to show $\D^-{V_{i,j}}\neq 0$ for all $(i,j)$. Suppose there exists $(i^*,j^*)$, $i^*\geq 1$, such that $\D^-{V_{i^*,j^*}}= 0$, then from 
$$
{H}^{\uparrow}\lc x_{i^*},y_{j^*},V_{i^*,j^*},\D^+{V_{i^*,j^*}}\rc-{H}_{\min}\lc x_{i^*},y_{j^*},V_{i^*,j^*}\rc\geq 0
$$
 and
${H}^{\downarrow} \lc x_{i^*},y_{j^*},V_{i^*,j^*},0\rc=0$
we deduce that 
\beq\label{Eq: num Hgeq0}
\boldsymbol{H}\lc x_{i^*},y_{j^*},V_{i^*,j^*},\D^+{V_{i^*,j^*}},\D^-{V_{i^*,j^*}}\rc\geq 0.
\eeq
By multiplying the equation \cref{Eq: D-HJB H} at ${i^*,j^*}$ and ${i^*,\bar \jmath^*}$ respectively by $\lambda_{\bar \jmath^*}$ and $\lambda_{j^*}$, summing up, we obtain
\begin{align*}
\frac{\rho}{\theta}\left(\lambda_{\bar \jmath^*}V_{i^*,j^*}+\lambda_{j^*}V_{i^*,\bar \jmath^*}\right)={}&\lambda_{j^*}\boldsymbol{H}\lc x_{i^*},y_{\bar \jmath^*},{V}_{i^*,\bar \jmath^*},\D^+{V_{i^*,\bar \jmath^*}},\D^-{V_{i^*,\bar \jmath^*}}\rc\\
{}&+\lambda_{\bar \jmath^*}\boldsymbol{H}\lc x_{i^*},y_{j^*},V_{i^*,j^*},\D^+{V_{i^*,j^*}},\D^-{V_{i^*,j^*}}\rc.
\end{align*}
Then, we obtain from \cref{Eq: num Hgeq0} and \cref{Prop: sol bound late} that
\begin{align*}
\lambda_{j^*}\boldsymbol{H}\lc x_{i^*},y_{\bar \jmath^*},{V}_{i^*,\bar \jmath^*},\D^+{V_{i^*,\bar \jmath^*}},\D^-{V_{i^*,\bar \jmath^*}}\rc\leq {}&\frac{\rho}{\theta}\left(\lambda_{\bar \jmath^*}V_{i^*,j^*}+\lambda_{j^*}V_{i^*,\bar \jmath^*}\right)\\
 \leq {}&\frac{\rho}{\theta}\left(\lambda_1+\lambda_2\right)\frac{(b(x_{i^*}+y_2/r))^{1-\gamma}}{1-\gamma}<0,
\end{align*}
hence $\D^-{V_{i^*,\bar \jmath^*}}>0$. Subtracting the following two equations,
\begin{align*}
\frac{\rho}{\theta} V_{i^*,j^*}
={}&\boldsymbol{H}\lc x_{i^*},y_{j^*},V_{i^*,j^*},\D^+{V_{i^*,j^*}},\D^-{V_{i^*,j^*}}\rc
+\lambda_{j^*}\lc V_{i^*,\bar \jmath^*}-V_{i^*,j^*} \rc,\\
\frac{\rho}{\theta} V_{i^*-1,j^*}
={}&\boldsymbol{H}\lc x_{i^*-1},y_{j^*},V_{i^*-1,j^*},\D^+{V_{i^*-1,j^*}},\D^-{V_{i^*-1,j^*}}\rc\\
&+\lambda_{j^*}\lc V_{i^*-1,\bar \jmath^*}-V_{i^*-1,j^*} \rc,
\end{align*}
and taking into account that $\D^-{V_{i^*,j^*}}= 0$, we obtain
\beq\label{Eq: H-H I}
\begin{aligned}
-\lambda_{j^*}\D x\D^-{V_{i^*,\bar \jmath^*}}={}&\boldsymbol{H}\lc x_{i^*},y_{j^*},V_{i^*,j^*},\D^+{V_{i^*,j^*}},\D^-{{V}_{i^*,j^*}}\rc\\
{}&-\boldsymbol{H}\lc x_{i^*-1},y_{j^*},{V}_{i^*-1,j^*},\D^+{V_{i^*-1,j^*}},\D^-{V_{i^*-1,j^*}}\rc.
\end{aligned}
\eeq
For brevity, let us name ${RHS}$ the right hand side of \cref{Eq: H-H I}. We are going to see ${RHS}\geq 0$. If $1<i^*<I$, then
\begin{align*}
&{RHS} \\
={}&{H}^{\downarrow} \lc x_{i^*},y_{j^*},V_{i^*,j^*},\D^-{V_{i^*,j^*}}\rc-{H}^{\downarrow} \lc x_{i^*-1},y_{j^*},V_{i^*-1,j^*},\D^-{V_{i^*-1,j^*}}\rc \tag{$I.1$}\\
{}&+{H}^{\uparrow} \lc x_{i^*},y_{j^*},V_{i^*,j^*},\D^+{V_{i^*,j^*}}\rc-{H}_{\min}\lc x_{i^*},y_{j^*},V_{i^*,j^*}\rc
\tag{$I.2$}\\
{}&-{H}^{\uparrow} \lc x_{i^*-1},y_{j^*},V_{i^*-1,j^*},\D^+{V_{i^*-1,j^*}}\rc+{H}_{\min}\lc x_{i^*-1},y_{j^*},V_{i^*-1,j^*}\rc \tag{$I.3$}
\end{align*}
From $\D^-{V_{i^*,j^*}}=0\leq \D^-{V_{i^*-1,j^*}}$ and \cref{Eq: Hdown leq 0}, we deduce 
$$
{H}^{\downarrow} \lc x_{i^*},y_{j^*},V_{i^*,j^*},\D^-{V_{i^*,j^*}}\rc=0,\quad {H}^{\downarrow} \lc x_{i^*-1},y_{j^*},V_{i^*-1,j^*},\D^-{V_{i^*-1,j^*}}\rc\leq 0,
$$
 hence $(I.1)\geq 0$. It is obvious $(I.2)\geq 0$, and we deduce from $\D^+{V_{i^*-1,j^*}}=\D^-{V_{i^*,j^*}}=0$ that $(I.3)= 0$. This yields ${RHS}\geq 0$.\par 
 If $i^*=1$, then $\D^-{V_{1,j^*}}=\D^+{V_{0,j^*}}=(V_{1,j^*}-V_{0,j^*})/\D x =0$, and 
 \begin{align*}
{RHS} ={}&{H}^{\uparrow} \lc x_{1},y_{j^*},V_{1,j^*},\D^+{V_{1,j^*}}\rc-{H}_{\min}\lc x_1,y_{j^*},V_{1,j^*}\rc-{H}^{\uparrow} \lc \ux,y_{j^*},V_{0,j^*},0\rc\\
={}&{H}^{\uparrow} \lc x_{1},y_{j^*},V_{0,j^*},\D^+{V_{1,j^*}}\rc-{H}_{\min}\lc x_1,y_{j^*},V_{1,j^*}\rc-{H}_{\min}\lc \ux,y_{j^*},V_{0,j^*}\rc\geq 0.
\end{align*}
If $i^*=I$, then $\D^-{V_{I,j^*}}=(V_{I,j^*}-V_{I-1,j^*})/\D x =0$ and  
 \begin{align*}
{RHS} ={}&{H}^{\downarrow} \lc x_{I},y_{j^*},V_{I,j^*},0\rc-{H}^{\downarrow} \lc x_{I-1},y_{j^*},V_{I-1,j^*},\D^-{V_{I-1,j^*}}\rc\geq 0.
\end{align*}
 We conclude by observing ${RHS}\geq 0$, contradicts the fact that $\D^-{V_{i^*,\bar \jmath^*}}>0$.  
\end{proof}
\begin{corollary}\label{Cor: Dv positive}
For any $\D x$ and for any solution $V$ to \cref{Eq: D-HJB H}, there exists $\varsigma>0$ such that
$\D^-{V_{i,j}}\geq \varsigma$ for all $i\geq 1$ and $j\in \{1,2\}$. 
\end{corollary}

\begin{proposition}\label{Prop: discrete Dv bound}
Let $V$ be a solution to \cref{Eq: D-HJB H}. There exists a constant $C>0$ uniform in $\D x$ such that, for all $(i,j)$, $i\geq 1$ and $j\in \{1,2\}$, $\D^+{V_{i-1,j}}=\D^-{V_{i,j}}\leq C$.
\end{proposition}

\begin{proof} From \cref{Prop: sol bound late}, we know that the grid function $V$ is bounded from below by $\check{U}_{0,1}$ and from above by $\check{V}_{I,2}$. Hence, there exists a constant $\tilde{C}>0$, uniform with respect to $\D x$, such that any solution $V$ to \cref{Eq: D-HJB H} satisfies ${p}_{\min}\lc x_i,y_j,V_{i,j}\rc\leq \tilde{C}$. Therefore, it is not restrictive to assume that the constant in the statement of \cref{Prop: discrete Dv bound} is larger than $\tilde{C}$. This allows us to focus on the pairs of indices $(i,j)$, $0\leq i\leq I-1$ and $j\in \{1,2\}$, such that $\D^+{V_{i,j}}\geq \tilde{C}>{p}_{\min}\lc x_i,y_j,V_{i,j}\rc$, thus
$$
H\lc x_i,y_j,V_{i,j},\D^+{V_{i,j}}\rc={H}^{\uparrow} \lc x_i,y_j,V_{i,j},\D^+{V_{i,j}}\rc\leq \frac{\rho}{\theta} V_{i,j}-\lambda_j\lc V_{i,\bar \jmath}-V_{i,j} \rc.
$$
We conclude by using the above mentioned bound on $V$ and the coercivity of the Hamiltonian. 
\end{proof}

\subsection{A Howard-Newton algorithm}
With the $\varsigma$ bound in \cref{Cor: Dv positive}, we consider an iterative method for solving the regularized HJB equation \cref{Eq: D-HJB HR} instead of \cref{Eq: D-HJB H} since they have the same unique solution if $1/{\eps}>\rho^{\p}\varsigma^{-\p}\lc(1-\g)\check{V}_{I,2}\rc^{\frac{1-\g \p}{1-\g}}$ ($\check{V}$ is the supersolution defined in \cref{Prop:sub-super disc}). The  Howard-Newton iterative algorithm is described in \cref{HN algo}.
\begin{algorithm}[htbp]
\caption{Howard-Newton algorithm}
\label{HN algo}
\begin{algorithmic}[1]
\State Initialize 
$
c^{F,(0)}_{i,j}=rx_i+y_j,\,\,c^{B,(0)}_{i,j}=rx_i+y_j,
$
and set $n = 0$
\Repeat
\State Update the saving policy: 
$
s^{F,({n})}_{i,j}=rx_i+y_j-c^{F,({n})}_{i,j},\quad s^{B,({n})}_{i,j}=rx_i+y_j-c^{B,({n})}_{i,j}.
$ 
\State Update the value function by solving with Newton's method: 
\beq\label{Eq: Howard solve V}
\begin{aligned}
\frac{\rho}{\theta} V^{({n})}_{i,j}
={}& s^{F,({n})}_{i,j}\D^+{V^{({n})}_{i,j}}+s^{B,({n})}_{i,j}\D^-{V^{({n})}_{i,j}}+\mathcal{F}\lc c^{F,({n})}_{i,j},V^{({n})}_{i,j}\rc+\mathcal{F}\lc c^{B,({n})}_{i,j},V^{({n})}_{i,j}\rc\\
{}&-\mathcal{F}\lc \bar{c}_{j}(x_i),V^{({n})}_{i,j}\rc+\lambda_j\lc V^{({n})}_{i,\bar \jmath}-V^{({n})}_{i,j} \rc.
\end{aligned}
\eeq
\State Update the consumption policy: 
\beq\label{Eq: Howard-Newton P-update}
\left\{
\begin{aligned}
\,\,  &c^{F,({n}+1)}_{i,j}=\min\left\{\rho^{\p}\lc \D^+{V^{({n})}_{i,j}}\rc_+^{-\p}\lc(1-\g)V^{(n)}_{i,j}\rc^{\frac{1-\g \p}{1-\g}},\bar{c}_{j}(x_i)\right\}\,\text{for}\,\,x_i<\bar{x},\\
& c^{F,({n}+1)}_{I,j}=r\bar{x}+y_j,\\
&c^{B,({n}+1)}_{i,j}=\max\left\{\min\left\{1/\eps,\rho^{\p}\lc \D^-{V^{({n})}_{i,j}}\rc_+^{-\p}\lc(1-\g)V^{(n)}_{i,j}\rc^{\frac{1-\g \p}{1-\g}}\right\},\bar{c}_{j}(x_i)\right\}\\
 &\text{for}\,\,x_i>\ux,\quad c^{B,({n}+1)}_{0,j}=r\ux+y_j.
\end{aligned}
   \right.
\eeq
\State $n \gets n+ 1$
\Until{$\sum_j\left(\max_i\left\vert c^{F,({n})}_{i,j}-c^{F,({n-1})}_{i,j}\right\vert+\max_i\left\vert c^{B,({n})}_{i,j}-c^{B,({n-1})}_{i,j}\right\vert\right)<\text{Tol}_c$.}
\end{algorithmic}
\end{algorithm}

\begin{lemma}\label{Lemma: c max}
Consider the admissible set of controls
\[
\begin{aligned}
\mathcal{C}^{F,(n+1)}_{i,j}&=[0,\bar{c}_{j}(x_i)],\quad
\mathcal{C}^{B,(n+1)}_{i,j}&=\left[\bar{c}_{j}(x_i),1/\eps\right].
\end{aligned}
\]
Then for $n\geq 0$,
\[
\begin{aligned}
c^{F,({n}+1)}_{i,j}
 &=\underset{c\in \mathcal{C}^{F,(n+1)}_{i,j}}{\arg\max}
   \left\{-c\D^+{V^{({n})}_{i,j}}+\mathcal{F}\lc c,V^{({n})}_{i,j}\rc\right\},\\
c^{B,({n}+1)}_{i,j}
 &=\underset{c\in \mathcal{C}^{B,(n+1)}_{i,j}}{\arg\max}
   \left\{-c\D^-{V^{({n})}_{i,j}}+\mathcal{F}\lc c,V^{({n})}_{i,j}\rc\right\}.
\end{aligned}
\]
\end{lemma}
We will see in the proof of \cref{Thm: Howard-Newton} that the set for $\mathcal{C}^{B,(n+1)}_{i,j}$ is indeed nonempty for all $n$.  
\begin{theorem}\label{Thm: Howard-Newton}
The sequence $V^{(n)}$ generated by \cref{HN algo} converges to the unique solution of \cref{Eq: D-HJB H}. 
\end{theorem}
\begin{proof}
{\it{Step 1.}} We show that $V^{(n)}$ is bounded above uniformly w.r.t. $n$. From \cref{Eq: Howard solve V}, we know $(V^{(n)}_{i,1},V^{(n)}_{i,2})$ is a subsolution of \cref{Eq: D-HJB HR}. We infer from \cref{Prop: sub super HR2} and \cref{Prop: comparison} that $V^{(n)}_{i,j}\leq \check{V}_{i,j}$ for all $n$. This ensures that $\lc(1-\g)V^{(n)}_{i,j}\rc^{\frac{1-\g \p}{1-\g}}$ is always well defined and bounded. \par
{\it{Step 2.}} We claim that $V^{(0)}\preceq V^{(1)}$. Indeed, suppose by contradiction that 
\beq\label{Eq: Vn-Vn+1}
V_{i^*,j^*}^{(0)}- V^{(1)}_{i^*,j^*}=\max_{i,j}\{V^{(0)}_{i,j}-V^{(1)}_{i,j}\}=\delta>0.
\eeq
This implies
\beq\label{Eq: lVn}
\lambda_{j^*}\lc V^{({0})}_{i^*,\bar \jmath^*}-V^{({0})}_{i^*,j^*} \rc\leq \lambda_{j^*}\lc V^{(1)}_{i^*,\bar \jmath^*}-V^{(1)}_{i^*,j^*} \rc,
\eeq
\beq\label{Eq Dvn Dvn+1}
\D^+V^{({1})}_{i^*,j^*}\geq \D^+V^{(0)}_{i^*,j^*}\,\,{\text{if}}\,\,i^*<I, \quad \D^-V^{({1})}_{i^*,j^*}\leq \D^-V^{(0)}_{i^*,j^*}\,\,{\text{if}}\,\,i^*>0.
\eeq
We consider three cases.\par
\fbox{\it{Case 1:}} $0<i^*<I$. From the initialization of \cref{HN algo}, we know
\beq\label{Eq: Howard V0}
\begin{aligned}
\frac{\rho}{\theta} V^{({0})}_{i,j}
= \mathcal{F}\lc \bar{c}_{j}(x_i),V^{({0})}_{i,j}\rc+\lambda_j\lc V^{({0})}_{i,\bar \jmath}-V^{({0})}_{i,j} \rc.
\end{aligned}
\eeq
It is easy to verify that $(\check{U}_{0,1},\check{U}_{0,2})$ is a subsolution to \cref{Eq: Howard V0}, hence $V^{({0})}_{i,j}\geq \check{U}_{0,j}$. We deduce from \cref{Eq: Howard solve V}, \cref{Eq: Howard-Newton P-update} and \cref{Lemma: c max} that
\beq\label{Eq: Howard Vn}
\begin{aligned}
\frac{\rho}{\theta} V^{({0})}_{i^*,j^*}
\leq{}& s^{F,({1})}_{i^*,j^*}\D^+{V^{({0})}_{i^*,j^*}}+s^{B,({1})}_{i^*,j^*}\D^-{V^{({0})}_{i^*,j^*}}+\mathcal{F}\lc c^{F,({1})}_{i^*,j^*},V^{({0})}_{i^*,j^*}\rc+\mathcal{F}\lc c^{B,({1})}_{i^*,j^*},V^{({0})}_{i^*,j^*}\rc\\
{}&-\mathcal{F}\lc \bar{c}_{j}(x_{i^*}),V^{({0})}_{i^*,j^*}\rc+\lambda_{j^*}\lc V^{({0})}_{i^*,\bar \jmath^*}-V^{({0})}_{i^*,j^*} \rc.
\end{aligned}
\eeq
From \cref{Eq: Vn-Vn+1} and same argument as in \cref{Lemma: num H_V},
\beq\label{Eq: Fvn}
\begin{aligned}
&\mathcal{F}\lc c^{F,({1})}_{i^*,j^*},V^{({0})}_{i^*,j^*}\rc+\mathcal{F}\lc c^{B,({1})}_{i^*,j^*},V^{({0})}_{i^*,j^*}\rc-\mathcal{F}\lc \bar{c}_{j}(x_{i^*}),V^{({0})}_{i^*,j^*}\rc\\
\leq {}&\mathcal{F}\lc c^{F,({1})}_{i^*,j^*},V^{({1})}_{i^*,j^*}\rc+\mathcal{F}\lc c^{B,({1})}_{i^*,j^*},V^{({1})}_{i^*,j^*}\rc-\mathcal{F}\lc \bar{c}_{j}(x_{i^*}),V^{({1})}_{i^*,j^*}\rc.
\end{aligned}
\eeq
By construction, $s^{F,({1})}_{i^*,j^*}\geq 0$ and $s^{B,({1})}_{i^*,j^*}\leq 0$. Then \cref{Eq Dvn Dvn+1} implies that
\beq\label{Eq: sDvn}
s^{F,({1})}_{i^*,j^*}\D^+{V^{(0)}_{i^*,j^*}}+s^{B,({1})}_{i^*,j^*}\D^-{V^{(0)}_{i^*,j^*}}\leq s^{F,({1})}_{i^*,j^*}\D^+{V^{(1)}_{i^*,j^*}}+s^{B,({1})}_{i^*,j^*}\D^-{V^{(1)}_{i^*,j^*}}.
\eeq
From \cref{Eq: lVn}, \cref{Eq: Fvn}, \cref{Eq: sDvn} and the equation \cref{Eq: Howard solve V} for $V^{(1)}$, we get $\rho \delta/\theta\leq 0$, a contradiction. \par
\fbox{\it{Case 2:}} $i^*=0$. The same argument holds, provided \cref{Eq: Howard Vn} is replaced with
$$
\begin{aligned}
\frac{\rho}{\theta} V^{(0)}_{0,j^*}
\leq s^{F,({1})}_{0,j^*}\D^+{V^{(0)}_{0,j^*}}+\mathcal{F}\lc c^{F,({1})}_{0,j^*},V^{(0)}_{0,j^*}\rc+\lambda_{j^*}\lc V^{(1)}_{0,\bar \jmath^*}-V^{(1)}_{0,j^*} \rc.
\end{aligned}
$$
The conclusion follows similarly as in the previous case. \par
\fbox{\it{Case 3:}} $i^*=I$. The same argument holds, provided \cref{Eq: Howard Vn} is replaced with
$$
\begin{aligned}
\frac{\rho}{\theta} V^{(0)}_{I,j^*}
\leq s^{B,({1})}_{I,j^*}\D^-{V^{(0)}_{I,j^*}}+\mathcal{F}\lc c^{B,({1})}_{I,j^*},V^{(0)}_{I,j^*}\rc+\lambda_{j^*}\lc V^{(1)}_{I,\bar \jmath^*}-V^{(1)}_{I,j^*} \rc,
\end{aligned}
$$
and the desired result follows.\par
{\it{Step 3.}} Following the proof of {\it{Step 2}} , we can obtain by induction $V^{(n)}\preceq V^{(n+1)}$.\par
Therefore, $V^{(n)}$ is an increasing sequence w.r.t. $n$ and is bounded above uniformly in $n$, yielding the desired result.
\end{proof}

\subsection{Convergence rate to the unique viscosity solution when \texorpdfstring{$\theta \geq 1$}{theta >= 1}}\label{Subsec: rate}
We denote by $(v_1,v_2)$ the unique solution of \cref{HJB}, and $(V_1,V_2)$ the unique solution of the discrete problem \cref{Eq: D-HJB H} on $ \mathcal{G}^{\D x}$.
\begin{lemma}\label{Lemma: error FD}
We make the same assumptions as in \cref{Prop: saving}, then there exists a constant $C$ (uniform in $\D x$) such that
\beq\label{Eq: FD erro v2}
\max\lc\max_{\,x\in \mathcal{G}^{\D x},\,x>\ux}\left\{\D^-{v_{2}(x)}-Dv_{2}(x)\right\},\,\max_{\,x\in \mathcal{G}^{\D x},\,x<\bar{x}}\left\{Dv_{2}(x)-\D^+{v_{2}(x)}\right\}\rc\leq C\D x.
\eeq
There exists $\eta$ sufficiently small (uniformly in $\D x$) such that
\beq\label{Eq: FD boundDx}
\max_{\,x\in \mathcal{G}^{\D x},\,\ux<x<\eta}\left\{\D^-{v_{1}(x)}-Dv_{1}(x)\right\}\leq C\sqrt{\D x},
\eeq
\beq\label{Eq: s1FD boundDx}
\max_{\,x\in \mathcal{G}^{\D x},\,\ux<x<\eta}\left\{-s_1(x)\lc\D^-{v_{1}(x)}-Dv_{1}(x)\rc\right\}\leq C\D x.
\eeq
\end{lemma}
\begin{proof}
In the proof, the constant $C$ may change from line to line but is always independent of $\D x$. From \cref{Prop: v sol}, we know that $v_j\in W^{2,\infty}_{loc}(\ux,\bar{x}]$. Since $s_2(\ux)>0$,  $D^2 v_2(\ux)$ is bounded. This implies $D^2 v_2\in L^\infty(\ux,\bar{x})$ and there exists $C_2$ (independent of $\D x$) such that 
$$
\begin{aligned}
\max\lc\max_{x\in \mathcal{G}^{\D x},\,x>\ux}\left\{\D^-{v_{2}(x)}-Dv_{2}(x)\right\},\,\max_{x\in \mathcal{G}^{\D x},\,x<\bar{x}}\left\{Dv_{2}(x)-\D^+{v_{2}(x)}\right\}\rc\leq C_2\Dx.
\end{aligned}
$$
To prove \cref{Eq: FD boundDx} and \cref{Eq: s1FD boundDx}, we consider two cases: $x=\ux+\D x$ and $\ux+\D x<x<\eta$.\par
Case 1:  $x=\ux+\D x$. From the strict concavity of $v_1$, 
$$
0<\D^-{v_{1}(\ux+\D x)}-Dv_{1}(\ux+\D x)<Dv_{1}(\ux)-Dv_{1}(\ux+\D x)\leq C\sqrt{\D x},
$$ 
where for the last inequality we used \cref{Eq: Dv1 ux}. From \cref{Eq: -s1 bound}, we know $0<-s_1(\ux+\D x)\leq C\sqrt{\D x}$, hence
$
-s_1(\ux+\D x)\lc\D^-{v_{1}(\ux+\D x)}-Dv_{1}(\ux+\D x)\rc \leq C\D x.
$\par
 Case 2: $\ux+2\D x\leq x< \eta$. By using a Taylor expansion with integral remainder, we deduce
$$
\D^-{v_{1}(x)}-Dv_{1}(x)=-\frac{1}{\D x}\int_{x-\D x}^{x}(z-(x-\D x))D^2v_1(z)dz.
$$
From $x-\D x\leq z<\eta$ and \cref{Prop: D2v1 ux}, we deduce that $0<-D^2v_1(z)<\frac{C}{\sqrt{z-\ux}}\leq \frac{C}{\sqrt{x-\D x-\ux}}\leq \frac{C}{\sqrt{\D x}}$, where for the last inequality we used $x\geq \ux+2\D x$. Hence,  
$$
\D^-{v_{1}(x)}-Dv_{1}(x)\leq \frac{C}{(\D x)^{1/2}}\frac{1}{\D x}\int_{x-\D x}^{x}(z-(x-\D x))dz\leq C(\D x)^{1/2}.
$$
Since $0<-s_1(x)\leq C\sqrt{x-\ux}$ for $\ux+2\D x\leq x< \eta$, 
$$
\begin{aligned}
-s_1(x)\lc\D^-{v_{1}(x)}-Dv_{1}(x)\rc={}&\frac{-s_1(x)}{\D x}\int_{x-\D x}^{x}(z-(x-\D x))\lc-D^2v_1(z)\rc dz\\
\leq {}&\frac{C\sqrt{x-\ux}}{\sqrt{x-\D x-\ux}}\frac{1}{\D x}\int_{x-\D x}^{x}(z-(x-\D x))dz\\
\leq {}&\frac{C(\sqrt{x-\D x-\ux}+\sqrt{\D x})}{\sqrt{x-\D x-\ux}}\D x\leq C\D x.
\end{aligned}
$$
For the last inequality, we observe that $\frac{\sqrt{\D x}}{\sqrt{x-\D x-\ux}}\leq 1$ since $x\geq \ux+2\D x$. 

\end{proof}
\begin{theorem}\label{Thm: convergence rate}
We make the same assumptions as in \cref{Prop: saving}, and $\bar{x}\in \mathcal{G}^{\D x}$ is sufficiently large such that $s_2(\bar{x})<0$. Then for sufficiently small $\D x$, there exists a constant $C>0$, independent of $\D x$, such that
\beq\label{Eq: error}
\max_{j,\,x\in \mathcal{G}^{\D x}} \left\vert v_j(x)-V_{j}(x)\right\vert \leq C\D x.
\eeq
\end{theorem}
\begin{proof} In this proof, $C$ denotes a generic positive constant that may change from line to line. \par
{\it{Step 1.}}
Let us define 
\beq\label{Eq: def sigma}
\sigma=\max_{j,\, x\in \mathcal{G}^{\D x}} \{v_j(x)-V_{j}(x)\}=v_{j^*}(x^*)-V_{j^*}(x^*),
\eeq
and suppose $\sigma>0$. Since $v_j\in C^1(\ux,+\infty)$, we know that if $x^*>\ux$ then
\beq\label{Eq: v(x) sub}
\frac{\rho}{\theta} v_{j^*}(x^*)=
	H\lc x^*,y_{j^*},v_{j^*}(x^*),Dv_{j^*}(x^*)\rc+\lambda_{j^*}(v_{\bar \jmath^*}(x^*)-v_{j^*}(x^*)).
\eeq
Since $Dv_j$ is uniformly continuous in $[\ux,+\infty)$, \cref{Eq: v(x) sub} holds if $x^*=\ux$. \par 
Since 
$
V_{j^*}(x^*+\D x)-V_{j^*}(x^*)\geq v_{j^*}(x^*+\D x)-v_{j^*}(x^*)
$
if $\ux\leq x^*<\bar{x}$, $\D^+{V_{j^*}(x^*)}\geq \D^+{v_{j^*}(x^*)}$ if $\ux\leq x^*<\bar{x}$. 
Similarly, $\D^-{V_{j^*}(x^*)}\leq \D^-{v_{j^*}(x^*)}$ if $\ux< x^*\leq \bar{x}$. From \cref{Lemma: H monotone},
\beq\label{Eq: HVHv}
\begin{aligned}
&\boldsymbol{H}\lc x^*,y_{j^*},V_{j^*}(x^*),\D^+{V_{j^*}(x^*)},\D^-{V_{j^*}(x^*)}\rc\\
\geq {}&\boldsymbol{H}\lc x^*,y_{j^*},V_{j^*}(x^*),\D^+{v_{j^*}(x^*)},\D^-{v_{j^*}(x^*)}\rc\\
\geq {}&\boldsymbol{H}\lc x^*,y_{j^*},v_{j^*}(x^*),\D^+{v_{j^*}(x^*)},\D^-{v_{j^*}(x^*)}\rc.
\end{aligned}
\eeq
For the last inequality, we used $\sigma>0$ and \cref{Lemma: num H_V}. Observe that if $x^*=\ux$ or $x^*=\bar{x}$, \cref{Eq: HVHv} holds with $\boldsymbol{H}$ given by \cref{Eq: H num}. We infer from \cref{Eq: D-HJB H} and \cref{Eq: HVHv} that
\beq\label{Eq: V(z) super}
\frac{\rho}{\theta}V_{j^*}(x^*)\geq \boldsymbol{H}\lc x^*,y_{j^*},v_{j^*}(x^*),\D^+{v_{j^*}(x^*)},\D^-{v_{j^*}(x^*)}\rc+\lambda_{j^*}(V_{\bar \jmath^*}(x^*)-V_{j^*}(x^*)).
\eeq
From \cref{Eq: def sigma}, $v_{j^*}(x^*)-V_{j^*}(x^*)\geq v_{\bar \jmath^*}(x^*)-V_{\bar \jmath^*}(x^*)$. Therefore,  by subtracting  \cref{Eq: v(x) sub} from \cref{Eq: V(z) super} we obtain 
\beq\label{Eq: v-V}
\begin{aligned}
&\frac{\rho}{\theta}\lc v_{j^*}(x^*)-V_{j^*}(x^*)\rc \\
\leq {}& H\lc x^*,y_{j^*},v_{j^*}(x^*),Dv_{j^*}(x^*)\rc-\boldsymbol{H}\lc x^*,y_{j^*},v_{j^*}(x^*),\D^+{v_{j^*}(x^*)},\D^-{v_{j^*}(x^*)}\rc.
\end{aligned}
\eeq
If $j^*=2$, then from \cref{Eq: FD erro v2} and the consistency of the numerical Hamiltonian $\boldsymbol{H}$, it follows $\frac{\rho}{\theta}\lc v_{2}(x^*)-V_{2}(x^*)\rc\leq C\D x$.\par 
We now estimate the right hand side of \cref{Eq: v-V}, with $j^*=1$, by making out three cases. \par 
{\fbox{\it{Case 1.}}} $\ux<x^*<\eta$, where $\eta$ has been introduced in \cref{Lemma: error FD}.
Since $s_1(x^*)<0$, we know $Dv_{1}(x^*)<p_{\min}\lc x^*,y_{1},v_{1}(x^*)\rc$, and $\D^-{v_{1}(x^*)}<p_{\min}\lc x^*,y_{1},v_{1}(x^*)\rc$ follows from the continuity of $Dv_{1}$, and 
$$
\begin{aligned}
&H\lc x^*,y_{1},v_{1}(x^*),Dv_{1}(x^*)\rc-\boldsymbol{H}\lc x^*,y_{1},v_{1}(x^*),\D^+{v_{1}(x^*)},\D^-{v_{1}(x^*)}\rc \\
= {}& H\lc x^*,y_{1},v_{1}(x^*),Dv_{1}(x^*)\rc-H\lc x^*,y_{1},v_{1}(x^*),\D^-{v_{1}(x^*)}\rc\\
\leq {}& (rx^*+y_{1})\lc Dv_{1}(x^*)-\D^-{v_{1}(x^*)}\rc\\
+&\underbrace{\frac{\rho^{\p}\lc Dv_{1}(x^*)\rc^{1-\p}}{\p-1} ((1-\g){v_{1}(x^*)})^{\frac{1-\g \p}{1-\g}}-\frac{\rho^{\p} \lc \D^-{v_{1}(x^*)}\rc^{1-\p}}{\p-1}((1-\g){v_{1}(x^*)})^{\frac{1-\g \p}{1-\g}}}_{(I)}
\end{aligned}
$$
From the mean value theorem, there exists $\xi\in (0,1)$ such that 
$$
\begin{aligned}
&(I)\\
={}&\lc-\rho^{\p}\lc \xi Dv_{1}(x^*)+(1-\xi)\D^-{v_{1}(x^*)}\rc^{-\psi}\rc ((1-\g){v_{1}(x^*)})^{\frac{1-\g \p}{1-\g}}\cdot\\
&\lc Dv_{1}(x^*)-\D^-{v_{1}(x^*)}\rc\\
={}& -\rho^{\p}\lc Dv_{1}(x^*)\rc^{-\psi}((1-\g){v_{1}(x^*)})^{\frac{1-\g \p}{1-\g}} \lc Dv_{1}(x^*)-\D^-{v_{1}(x^*)}\rc\\
&+\rho^{\p}\underbrace{\lc\lc \xi Dv_{1}(x^*)+(1-\xi)\D^-{v_{1}(x^*)}\rc^{-\psi}-\lc Dv_{1}(x^*)\rc^{-\psi}\rc}_{(II)}\underbrace{((1-\g){v_{1}(x^*)})^{\frac{1-\g \p}{1-\g}}}_{III}\cdot\\
&\lc \D^-{v_{1}(x^*)}-Dv_{1}(x^*)\rc\\
\leq {}& c_1(x^*)\lc \D^-{v_{1}(x^*)}-Dv_{1}(x^*)\rc\\
&+\rho^{\p}\lc Dv_{1}(\bar{x})\rc^{-\psi-1}\lc b\lc \bar{x}+\frac{y_2}{r}\rc\rc^{1-\g \p}\lc \D^-{v_{1}(x^*)}-Dv_{1}(x^*)\rc^2,
\end{aligned}
$$
where we used \cref{c optimal} for the last inequality. To bound $(II)$ for the last inequality, we infer from the concavity of $v_1$ that $\D^-{v_{1}(x^*)}>Dv_{1}(x^*)>Dv_{1}(\bar{x})>0$, hence  
$$
(II)\leq \psi \lc Dv_{1}(\bar{x})\rc^{-\psi-1}\lc \D^-{v_{1}(x^*)}-Dv_{1}(x^*)\rc.
$$
For $(III)$, we have used \cref{Prop:sub-super} and the comparison principle. \par
Since $s_1(x^*)=rx^*+y_1-c_1(x^*)$,
$$
\begin{aligned}
&H\lc x^*,y_{1},v_{1}(x^*),Dv_{1}(x^*)\rc-\boldsymbol{H}\lc x^*,y_{1},v_{1}(x^*),\D^+{v_{1}(x^*)},\D^-{v_{1}(x^*)}\rc \\
\leq {}&-s_1(x^*)\lc \D^-{v_{1}(x^*)}-Dv_{1}(x^*)\rc+C\lc \D^-{v_{1}(x^*)}-Dv_{1}(x^*)\rc^2 \leq C\D x,
\end{aligned}
$$
where for the last inequality we used \cref{Eq: FD boundDx} and \cref{Eq: s1FD boundDx} in \cref{Lemma: error FD}.\par
\noindent \fbox{\it{Case 2.}} $x^*=\ux$. Since $s_1(\ux)=0$, $Dv_{1}(\ux)=p_{\min}\lc \ux,y_{1},v_{1}(\ux)\rc$ and $\D^+v_{1}(\ux)<p_{\min}\lc \ux,y_{1},v_{1}(\ux)\rc$, hence 
$
H\lc \ux,y_{1},v_{1}(\ux),Dv_{1}(\ux)\rc-{H}^{\uparrow}\lc \ux,y_{1},v_{1}(\ux),\D^+{v_{1}(\ux)}\rc=0.
$\par 

\noindent  \fbox{\it{Case 3.}} $x^*\in [\eta,\bar{x}]$. Since $s_{1}(x^*)<0$, 
$$
\begin{aligned}
&H\lc x^*,y_{1},v_{1}(x^*),Dv_{1}(x^*)\rc-\boldsymbol{H}\lc x^*,y_{1},v_{1}(x^*),\D^+{v_{1}(x^*)},\D^-{v_{1}(x^*)}\rc \\
= {}& H\lc x^*,y_{1},v_{1}(x^*),Dv_{1}(x^*)\rc-H\lc x^*,y_{1},v_{1}(x^*),\D^-{v_{1}(x^*)}\rc \leq C\D x
\end{aligned}
$$
For the last inequality, we used $v_1\in W^{2,\infty}(\eta,\bar{x})$.\par 
Since we argue with sufficiently small $\D x$, we may assume $\D x<1$ without loss of generality. For all the three cases above, we can use \cref{Eq: v-V} and finally obtain that
$
\sigma\leq  \frac{C\theta}{\rho} \D x$.\par

{\it{Step 2.}} Now we reverse the direction of the estimate and define
$$
\sigma_2=\sup_{j,\,\,x\in \mathcal{G}^{\D x}} \{V_j(x)-v_{j}(x)\}.
$$
The proof that $\sigma_2\leq  C \D x$ is similar and we omit the details. 
\end{proof}

\section{Solution method in the case \texorpdfstring{$0<\theta<1$}{0 < theta < 1}}\label{Sec: early}
Throughout this section we make the standing assumption $0<\theta<1$.
Consider the map $(V_1,V_2)=\Gamma_{\eps}(\widetilde{V}_1,\widetilde{V}_2)$ defined by
\beq\label{Eq: D-HJB Fix}
\left\{
      \begin{aligned}
\quad &\frac{\rho}{\theta} V_{i,1}
=\boldsymbol{H}_{\eps}\lc x_i,y_1,\widetilde{V}_{i,1},\D^+{V_{i,1}},\D^-{V_{i,1}}\rc
+\lambda_1\lc V_{i,2}-V_{i,1} \rc,\\
&\frac{\rho}{\theta} V_{i,2}
=\boldsymbol{H}_{\eps}\lc x_i,y_2,\widetilde{V}_{i,2},\D^+{V_{i,2}},\D^-{V_{i,2}}\rc
+\lambda_2\lc V_{i,1}-V_{i,2} \rc.
     \end{aligned}
   \right.
\eeq
Any fixed point of $\Gamma_{\eps}$ is a solution to the regularized HJB equation \cref{Eq: D-HJB HR}.
The following comparison principle is the discrete form of \cref{Prop: comparison fixed}. 
\begin{proposition}\label{Prop: comparison eps}
Assume $\check{U}\preceq \widetilde{V}\preceq \check{V}$. Let ${U}$ and ${V}$ be respectively a sub- and supersolution of \cref{Eq: D-HJB Fix}, then ${U}\preceq {V}$.
\end{proposition}
\begin{proof}
We argue by contradiction. Suppose 
$
{U}_{i^*,j^*}-{V}_{i^*,j^*}=\max_{i,j}\{{U}_{i,j}-{V}_{i,j}\}=\delta>0.
$
We denote $\bar \jmath^*=3-j^*$ and observe
$
{U}_{i^*,j^*}-{V}_{i^*,j^*}\geq {U}_{i^*,\bar \jmath^*}-{V}_{i^*,\bar \jmath^*},
$
hence
\beq\label{Eq: l(U-V)}
\lambda_{j^*}({U}_{i^*,\bar \jmath^*}-{U}_{i^*,j^*})\leq \lambda_{j^*}({V}_{i^*,\bar \jmath^*}-{V}_{i^*,j^*}).
\eeq
Since 
$
{U}_{i^*+1,j^*}-{V}_{i^*+1,j^*}\leq {U}_{i^*,j^*}-{V}_{i^*,j^*},\quad {U}_{i^*-1,j^*}-{V}_{i^*-1,j}\leq {U}_{i^*,j^*}-{V}_{i^*,j^*},
$
we infer that
$
\D^+{{V}_{i^*,j^*}}\geq \D^+{{U}_{i^*,j^*}}$, $\D^-{{V}_{i^*,j^*}}\leq \D^-{{U}_{i^*,j^*}}$. From \cref{Lemma: H monotone} we deduce 
\beq\label{Eq: HU HV0}
\begin{aligned}
&\boldsymbol{H}_{\eps}\lc x_{i^*},y_{j^*},\widetilde{V}_{i^*,j^*},\D^+{{U}_{i^*,j^*}},\D^-{{U}_{i^*,j^*}}\rc\\
\leq {}&\boldsymbol{H}_{\eps}\lc x_{i^*},y_{j^*},\widetilde{V}_{i^*,j^*},\D^+{{V}_{i^*,j^*}},\D^-{{V}_{i^*,j^*}}\rc.
\end{aligned}
\eeq
and 
\beq\label{Eq: compare U}
\frac{\rho}{\theta} {U}_{i^*,j^*}\leq \boldsymbol{H}_{\eps}\lc  x_{i^*},y_{j^*},\widetilde{V}_{i^*,j^*},\D^+{{V}_{i^*,j^*}},\D^-{{V}_{i^*,j^*}}\rc+\lambda_{j^*}({U}_{i^*,\bar \jmath^*}-{U}_{i^*,j^*}).
\eeq
Since ${V}$ is a supersolution,
\beq\label{Eq: compare V}
\frac{\rho}{\theta}{V}_{i^*,j^*}\geq \boldsymbol{H}_{\eps}\lc  x_{i^*},y_{j^*},\widetilde{V}_{i^*,j^*},\D^+{{V}_{i^*,j^*}},\D^-{{V}_{i^*,j^*}}\rc+\lambda_{j^*}({V}_{i^*,\bar \jmath^*}-{V}_{i^*,j^*}).
\eeq
Subtracting \cref{Eq: compare V} from \cref{Eq: compare U} and using \cref{Eq: l(U-V)}, we get $\rho \delta/\theta\leq 0$, a contradiction. 
\end{proof}

We observe that, since the value $\widetilde{V}_{i,j}$ is fixed in \cref{Eq: D-HJB Fix}, \cref{Prop: comparison eps} holds for all $\theta>0$. A major difficulty in the case $\theta<1$ is that \cref{Prop: comparison} may not hold. Therefore, we cannot directly use the comparison principle in \cref{Prop: comparison} and obtain $\check{U}\preceq {V}\preceq \check{V}$, for all solution $V$ to \cref{Eq: D-HJB H}, as in the case $\theta\geq 1$. However, the following results will allow us to look for a solution $V$ such that $\check{U}\preceq {V}\preceq \check{V}$. 
\begin{lemma}\label{Lemma: early barrier}
We have the barrier properties
\begin{itemize}
\item[(i).] If $({U}_1,{U}_2)=\Gamma_{\eps}(\check{U}_1,\check{U}_2)$, then
$
{U}\succeq \check{U}.
$
\item[(ii).] If $({V}_1,{V}_2)=\Gamma_{\eps}(\check{V}_1,\check{V}_2)$, then
$
{V}\preceq \check{V}.
$
\end{itemize}
\end{lemma}
\begin{proof}
(i). If $\widetilde{V}_{i,j}=\check{U}_{i,j}$, $(\check{U}_1,\check{U}_2)$ itself is a subsolution to the HJB equation \cref{Eq: D-HJB Fix}. From \cref{Prop: comparison eps}, ${U}_{i,j}\geq \check{U}_{i,j}$. \par
(ii). Similarly, if $\widetilde{V}_{i,j}=\check{V}_{i,j}$ then $(\check{V}_1,\check{V}_2)$ itself is a supersolution to the HJB equation \cref{Eq: D-HJB Fix}. From \cref{Prop: comparison eps}, ${V}_{i,j}\leq \check{V}_{i,j}$. 
\end{proof}

\begin{proposition}\label{Lemma: monotone fp}
The map $\Gamma_{\eps}$ is monotone: if $(U_1,U_2)=\Gamma_{\eps}(\widetilde{U}_1,\widetilde{U}_2)$, $(V_1,V_2)=\Gamma_{\eps}(\widetilde{V}_1,\widetilde{V}_2)$ and $\check{U}\preceq \widetilde{U}\preceq \widetilde{V}\preceq \check{V}$, then ${U}\preceq {V}$.
\end{proposition}
\begin{proof}
From $\check{U}\preceq \widetilde{U}\preceq \widetilde{V}\preceq \check{V}$ and \cref{Lemma: num H_Veps}, we know for all $(i,j)$,
$$
\boldsymbol{H}_{\eps}\lc x_{i},y_{j},\widetilde{U}_{i,j},\D^+{{U}_{i,j}},\D^-{{U}_{i,j}}\rc
\leq \boldsymbol{H}_{\eps}\lc x_{i},y_{j},\widetilde{V}_{i,j},\D^+{{U}_{i,j}},\D^-{{U}_{i,j}}\rc,
$$
hence $(U_1,U_2)$ is a subsolution to the discrete HJB system $(V_1,V_2)=\Gamma_{\eps}(\widetilde{V}_1,\widetilde{V}_2)$. \cref{Prop: comparison eps} yields the desired result. 
\end{proof}
We deduce the following invariance principle for $\Gamma_{\eps}$, analogous to \cref{Prop: invariance}. 
\begin{proposition}\label{Prop: invariance eps}
If $\check{U}\preceq \widetilde{V}\preceq \check{V}$ and $V=(V_1,V_2)=\Gamma_{\eps}(\widetilde{V}_1,\widetilde{V}_2)$, then $\check{U}\preceq {V}\preceq \check{V}$. 
\end{proposition}
\begin{proof}
Let us denote $\widehat{V}=(\widehat{V}_1,\widehat{V}_2)=\Gamma_{\eps}(\check{V}_1,\check{V}_2)$. From \cref{Lemma: monotone fp} and $\widetilde{V}\preceq \check{V}$, $V \preceq \widehat{V}$. From \cref{Lemma: early barrier} we obtain $V\preceq \widehat{V}\preceq \check{V}$. The proof for $\check{U}\preceq {V}$ is similar.
\end{proof}
\begin{definition}\label{Def: sol early}
The grid function $V^{\eps}$ is a solution to \cref{Eq: D-HJB HR} if the equations are satisfied and $\check{U}\preceq V^{\eps}\preceq \check{V}$. We say that $\underline{V}^{\eps}$ is the minimal solution of \cref{Eq: D-HJB HR} if it is a solution and $\underline{V}^{\eps}\preceq V^{\eps}$ for all $V^{\eps}$ solution to \cref{Eq: D-HJB HR}.
\end{definition}
\begin{corollary}\label{Cor: leq Vmin}
If $\check{U}\preceq \widetilde{V}\preceq \underline{V}^{\eps}$ and $(V_1,V_2)=\Gamma_{\eps}(\widetilde{V}_1,\widetilde{V}_2)$ then ${V}\preceq \underline{V}^{\eps}$.
\end{corollary}

Taking advantage of the theoretical analysis above, we propose \cref{HTK algo} to solve the discrete HJB equation in the case $0<\theta<1$. Let $\bar{c}_j$ be defined as in \cref{Eq: def cbar}. 

\begin{algorithm}[h!]
\caption{Howard-Tarski-Kantorovich algorithm: upward variant}
\label{HTK algo}
\begin{algorithmic}[1]
\State 
Initialize $V^{(0)}_{i,j}=\check{U}_{i,j}$, and set $k = 0$ \Comment{Outer loop}
\Repeat
   \State Initialize $V^{(k,0)}_{i,j}=V^{(k)}_{i,j}$, and set $n = 0$\\
    \Comment{Inner loop: Howard algorithm for solving $\lc V^{(k+1)}_1,V^{(k+1)}_2\rc=\Gamma_{\eps}
    \lc {V}^{(k)}_1,{V}^{(k)}_2\rc $}
   \Repeat
      \State Update the policy 
\beq\label{Eq: Howard-Tarski-Kantorovich P-update}
\left\{
\begin{aligned}
\,\, &c^{F,(k,{n})}_{i,j}=\min\left\{\rho^{\p}\lc \D^+{V^{(k,{n})}_{i,j}}\rc_+^{-\p}\lc(1-\g){V}^{(k)}_{i,j}\rc^{\frac{1-\g \p}{1-\g}},\bar{c}_{j}(x_i)\right\}\\
&\text{for}\,\,x_i<\bar{x},\quad
c^{F,(k,{n})}_{I,j}=r\bar{x}+y_j,\quad  s^{F,(k,{n})}_{i,j}=rx_i+y_j- c^{F,(k,{n})}_{i,j},\\
&c^{B,(k,{n})}_{i,j}=\max\left\{\min\left\{1/{\eps},\rho^{\p}\lc \D^-{V^{(k,{n})}_{i,j}}\rc_+^{-\p}\lc(1-\g){V}^{(k)}_{i,j}\rc^{\frac{1-\g \p}{1-\g}}\right\},\bar{c}_{j}(x_i)\right\}\\
& \text{for}\,\,x_i>\ux,\quad
 c^{B,(k,{n})}_{0,j}=r\ux+y_j,\quad  s^{B,(k,{n})}_{i,j}=rx_i+y_j- c^{B,(k,{n})}_{i,j}.
\end{aligned}
   \right.
\eeq
   \State Solve 
\beq\label{Eq: Howard inside Tarski solve V}
\begin{aligned}
\frac{\rho V^{(k,{n}+1)}_{i,j}}{\theta}
={}&s^{F,(k,{n})}_{i,j}\D^+{V^{(k,{n}+1)}_{i,j}}+s^{B,(k,{n})}_{i,j}\D^-{V^{(k,{n}+1)}_{i,j}}\\
&+\mathcal{F}\lc c^{F,(k,{n})}_{i,j},V^{(k)}_{i,j}\rc+\mathcal{F}\lc c^{B,(k,{n})}_{i,j},V^{(k)}_{i,j}\rc-\mathcal{F}\lc \bar{c}_{j}(x_i),V^{(k)}_{i,j}\rc\\
&+\lambda_j\lc V^{(k,{n}+1)}_{i,\bar \jmath}-V^{(k,{n}+1)}_{i,j} \rc.
\end{aligned}
\eeq
     \State $n\gets n+1$
     \Until{$n\ge 2$ \textbf{and} $\sum_j\left(\max_i\left\vert c^{F,(k,{n}-1)}_{i,j}-c^{F,(k,{n}-2)}_{i,j}\right\vert+\max_i\left\vert c^{B,(k,{n}-1)}_{i,j}-c^{B,(k,{n}-2)}_{i,j}\right\vert\right)<\text{Tol}_c$}
 \State Set $V^{(k+1)}_{i,j}=V^{(k,n)}_{i,j}$.
 \State $k\gets k+1$
\Until{$\sum_j\max_i\frac{\left\vert V^{(k)}_{i,j}-V^{(k-1)}_{i,j}\right\vert}{1+\left\vert V^{(k-1)}_{i,j}\right\vert}<\text{Tol}_V$.} 
\end{algorithmic}
\end{algorithm}

\begin{theorem}\label{Thm: Tarski convergence}
The sequence $V^{(k)}$ generated by \cref{HTK algo} converges to the minimal solution $\underline{V}^{\eps}$ of \cref{Eq: D-HJB HR}. 
\end{theorem}
\begin{proof}
{\it{Step 1.}} Solvability of the Howard inner loop. We first claim that for each $\left(V^{(k)}_{i,1},V^{(k)}_{i,2}\right)$ such that $\check{U}\preceq V^{(k)}\preceq\check{V}$, the sequence $V^{(k,n)}$ converges to $V^{(k+1)}$. The proof is similar to that of \cref{Thm: Howard-Newton}, hence we skip it.\par
{\it{Step 2.}} Uniform boundedness of the sequence $V^{(k)}$. From \cref{Prop: invariance eps}, we deduce that if $\check{U}\preceq V^{(k)}\preceq \check{V}$ then $\check{U}\preceq V^{(k+1)}\preceq \check{V}$. By induction, $\check{U}\preceq V^{(k)}\preceq \check{V}$ holds for all $k$. This ensures that at each outer iteration the term $\lc(1-\g){V}^{(k)}_{i,j}\rc^{\frac{1-\g \p}{1-\g}}$ is always well defined and bounded. \par 
{\it{Step 3.}} Monotonicity and convergence. We first observe that \cref{Lemma: early barrier} implies $V^{(1)}\succeq V^{(0)}$. From \cref{Lemma: monotone fp}, if $V^{(k)}\succeq V^{(k-1)}$, then $V^{(k+1)}\succeq V^{(k)}$. Arguing by induction, this implies that $V^{(k)}$ is increasing. Since $V^{(k)}\preceq \check{V}$ holds for all $k$, $V^{(k)}$ converges to a limit $\widehat{V}$ such that $\widehat{V}\preceq \check{V}$. \par
{\it{Step 4.}} Minimality. From \cref{Def: sol early}, we know if there exists a solution $V^{\eps}$, then ${V}^{(0)}=\check{U}\preceq V^{\eps}$. From \cref{Cor: leq Vmin}, an easy induction leads to ${V}^{(k)}\preceq V^{\eps}$ for all $k$, hence $\widehat{V}\preceq V^{\eps}$. From the design of the algorithm, $\widehat{V}$ is solution to \cref{Eq: D-HJB HR}, therefore $\widehat{V}=\underline{V}^{\eps}$. 
\end{proof}
\begin{remark}
In \cref{HN algo}, we only need a strictly negative upper bound so that
$
\lc(1-\g){V}^{(n)}_{i,j}\rc^{\frac{1-\g \p}{1-\g}}
$
is well defined, since $\lc(1-\g){V}^{(n)}_{i,j}\rc^{\frac{1-\g \p}{1-\g}}\to 0 $ if ${V}^{(n)}_{i,j}\to -\infty$.\par In \cref{HTK algo}, with $1-\g \psi<0$, we need both upper and lower bounds to ensure that $\lc(1-\g){V}^{(k)}_{i,j}\rc^{\frac{1-\g \p}{1-\g}}$ is well defined and bounded.
\end{remark}


The next result deals with the monotone behavior of the solution to \cref{Eq: D-HJB Fix} as $\eps$ decreases. 
\begin{proposition}\label{Prop: ordering eps}
Assume $\eps>\eps'>0$ and $\check{U}\preceq \widetilde{V}^{\eps}\preceq  \widetilde{V}^{\eps'}\preceq \check{V}$. Let ${V}^{\eps}$ and ${V}^{\eps'}$ be respectively the solutions to \cref{Eq: D-HJB Fix} with $(V^{\eps}_1,V^{\eps}_2)=\Gamma_{\eps}(\widetilde{V}^{\eps}_1,\widetilde{V}^{\eps}_2)$ and $(V^{\eps'}_1,V^{\eps'}_2)=\Gamma_{\eps'}(\widetilde{V}^{\eps'}_1,\widetilde{V}^{\eps'}_2)$. Then ${V}^{\eps}\preceq {V}^{\eps'}$.
\end{proposition}
\begin{proof}
Since $\eps>\eps'>0$ and $\check{U}_{i,j}\leq \widetilde{V}^{\eps}_{i,j}\leq \widetilde{V}^{\eps'}_{i,j}\leq \check{V}_{i,j}$,
$$
 \begin{aligned}
\boldsymbol{H}_{\eps}\lc x_i,y_j,\widetilde{V}^{\eps}_{i,j},\D^+{V^{\eps}_{i,j}},\D^-{V^{\eps}_{i,j}}\rc\leq {}&\boldsymbol{H}_{\eps'}\lc x_i,y_j,\widetilde{V}^{\eps}_{i,j},\D^+{V^{\eps}_{i,j}},\D^-{V^{\eps}_{i,j}}\rc\\
\leq {}&\boldsymbol{H}_{\eps'}\lc x_i,y_j,\widetilde{V}^{\eps'}_{i,j},\D^+{V^{\eps}_{i,j}},\D^-{V^{\eps}_{i,j}}\rc,
 \end{aligned}
$$
where we have used \cref{Lemma: num H_Veps} for the last inequality. Therefore, ${V}^{\eps}$ is a subsolution of the equation satisfied by ${V}^{\eps'}$, we apply \cref{Prop: comparison eps} and obtain ${V}^{\eps}\preceq {V}^{\eps'}$.
\end{proof}
Let us now deal with the existence of solution $\underline{V}$ to \cref{Eq: D-HJB H}, obtained as the limit of $\underline{V}^{\eps}$ as $\eps \to 0$.
\begin{proposition}\label{Prop: Tarski eps}
 If $\eps>\eps'>0$, then $\underline{V}^{\eps}\preceq \underline{V}^{\eps'}$. Moreover, $\underline{V}^{\eps}$ converges uniformly on $\mathcal{G}^{\D x}$ to the minimal solution of \cref{Eq: D-HJB H} as $\eps \to 0$. 
\end{proposition}
\begin{proof}
Let us denote by $V^{(k),\eps}$ the sequence of grid functions constructed by the Howard-Tarski-Kantorovich algorithm given the regularization parameter $\eps$. Observe that the initial guess does not depend on $\eps$, namely that $V^{(0),\eps}_{i,j}=V^{(0),\eps'}_{i,j}=\check{U}_{i,j}$. From \cref{Prop: ordering eps}, we infer that if $V^{(k),\eps}\preceq {V}^{(k),\eps'}$ then $V^{(k+1),\eps}\preceq {V}^{(k+1),\eps'}$. By induction, we obtain that $V^{(k),\eps}\preceq {V}^{(k),\eps'}$ for all $k\in \mathbb{N}$. Sending $k\to +\infty$ yields $\underline{V}^{\eps}\preceq \underline{V}^{\eps'}$. From \cref{Prop: invariance eps}, $\underline{V}^{\eps}\preceq \check{V}$ for all $\eps$. We thus conclude that the sequence $\underline{V}^{\eps}$ converges uniformly on $\mathcal{G}^{\D x}$ to the minimal solution of \cref{Eq: D-HJB H}. 
\end{proof}\par
We observe that, having established the existence of a minimal solution, we can use the same proof as that of \cref{Prop: positive Dv} to show $\D^-{\underline{V}_{i,j}}>\varsigma$ for all $i\geq1$ when $0<\theta<1$ . The regularized HJB equation \cref{Eq: D-HJB HR} and \cref{Eq: D-HJB H} have the same minimal solution if $1/{\eps}>\rho^{\p}\varsigma^{-\p}\lc(1-\g)\check{U}_{0,1}\rc^{\frac{1-\g \p}{1-\g}}$, where $\check{U}$ is defined in \cref{Prop:sub-super disc}. 
 \begin{proposition}
For $\eps>0$ sufficiently small, $\underline{V}^{\eps}=\underline{V}$. 
 \end{proposition}
 Next, we introduce \cref{HTK algo down}, which is a variant of \cref{HTK algo} in which the outer loop starts with $(\check{V}_1,\check{V}_2)$. Then, $V^{(k)}$ is a non increasing sequence of grid functions. 
 \begin{algorithm}[h!]
\caption{Howard-Tarski-Kantorovich algorithm: downward variant}
\label{HTK algo down}
\begin{algorithmic}[1]
\State Initialize $V^{(0)}_{i,j}=\check{V}_{i,j}$, and set $k = 0$ \Comment{Outer loop}
\Repeat
   \State Initialize $c^{F,(k,0)}_{i,j}=c^{B,(k,0)}_{i,j}=rx_i+y_j$ (null saving), and set $n = 0$\\
       \Comment{Inner loop: Howard algorithm for solving $\lc V^{(k+1)}_1,V^{(k+1)}_2\rc=\Gamma_{\eps}
    \lc {V}^{(k)}_1,{V}^{(k)}_2\rc$}
   \Repeat
   \State Policy evaluation: solve for $V^{(k,n)}$ the equation
\beq\label{Eq: Howard inside Tarski solve V down}
\begin{aligned}
\frac{\rho V^{(k,{n})}_{i,j}}{\theta}
={}&s^{F,(k,{n})}_{i,j}\D^+{V^{(k,{n})}_{i,j}}+s^{B,(k,{n})}_{i,j}\D^-{V^{(k,{n})}_{i,j}}\\
&+\mathcal{F}\lc c^{F,(k,{n})}_{i,j},V^{(k)}_{i,j}\rc+\mathcal{F}\lc c^{B,(k,{n})}_{i,j},V^{(k)}_{i,j}\rc-\mathcal{F}\lc \bar{c}_{j}(x_i),V^{(k)}_{i,j}\rc\\
&+\lambda_j\lc V^{(k,{n})}_{i,\bar \jmath}-V^{(k,{n})}_{i,j} \rc,
\end{aligned}
\eeq
where $s^{F,(k,n)}_{i,j}=rx_i+y_j-c^{F,(k,n)}_{i,j}$ and $s^{B,(k,n)}_{i,j}=rx_i+y_j-c^{B,(k,n)}_{i,j}$.
   \State Update the policy
\beq\label{Eq: Howard-Tarski-Kantorovich P-update down}
\left\{
\begin{aligned}
\,\, &c^{F,(k,{n}+1)}_{i,j}=\min\left\{\rho^{\p}\lc \D^+{V^{(k,{n})}_{i,j}}\rc_+^{-\p}\lc(1-\g){V}^{(k)}_{i,j}\rc^{\frac{1-\g \p}{1-\g}},\bar{c}_{j}(x_i)\right\}\\
&\text{for}\,\,x_i<\bar{x},\quad
c^{F,(k,{n}+1)}_{I,j}=r\bar{x}+y_j,\\
&c^{B,(k,{n}+1)}_{i,j}=\max\left\{\min\left\{1/{\eps},\rho^{\p}\lc \D^-{V^{(k,{n})}_{i,j}}\rc_+^{-\p}\lc(1-\g){V}^{(k)}_{i,j}\rc^{\frac{1-\g \p}{1-\g}}\right\},\bar{c}_{j}(x_i)\right\}\\
& \text{for}\,\,x_i>\ux,\quad
 c^{B,(k,{n}+1)}_{0,j}=r\ux+y_j.
\end{aligned}
   \right.
\eeq
      \State $n\gets n+1$
     \Until{$\sum_j\left(\max_i\left\vert c^{F,(k,{n})}_{i,j}-c^{F,(k,{n}-1)}_{i,j}\right\vert+\max_i\left\vert c^{B,(k,{n})}_{i,j}-c^{B,(k,{n}-1)}_{i,j}\right\vert\right)<\text{Tol}_c$}
 \State Set $V^{(k+1)}_{i,j}=V^{(k,n-1)}_{i,j}$.
 \State $k\gets k+1$
\Until{$\sum_j\max_i\frac{\left\vert V^{(k)}_{i,j}-V^{(k-1)}_{i,j}\right\vert}{1+\left\vert V^{(k-1)}_{i,j}\right\vert}<\text{Tol}_V$.}
\end{algorithmic}
\end{algorithm}

\begin{theorem}\label{Thm: Tarski convergence maximal}
The sequence $V^{(k)}$ generated by \cref{HTK algo down} converges to the maximal solution $\overline{V}^{\eps}$ of \cref{Eq: D-HJB HR}. 
\end{theorem}
The proof is similar to that of \cref{Thm: Tarski convergence}, hence we omit the details. Observe that the design of Howard iteration inner loop is different. In \cref{HTK algo}, $V^{(k,n)}$ is an increasing sequence for both indexes $k$ and $n$. Therefore, we initialize $V^{(k,0)}_{i,j}=V^{(k)}_{i,j}$ at each inner loop of \cref{HTK algo} in order to accelerate convergence. In \cref{HTK algo down}, $V^{(k)}$ is decreasing while $V^{(k,n)}$ is increasing in $n$ for each fixed $k$. Therefore, $V^{(k)}$ itself may no longer be a subsolution of $\lc V^{(k+1)}_1,V^{(k+1)}_2\rc=\Gamma_{\eps}
    \lc {V}^{(k)}_1,{V}^{(k)}_2\rc $ and we initialize the inner loop of \cref{HTK algo down} using a feasible consumption policy. \par
Similarly to \cref{Prop: Tarski eps}, we can obtain the maximal solution $\overline{V}$ as the limit of $\overline{V}^{\eps}$ when $\eps\to 0$. Due to the lack of comparison principle in the case $0<\theta<1$, we do not have uniqueness results for \cref{HJB} or \cref{Eq: D-HJB H}. Nevertheless, in all our experiments. we have observed that \cref{HTK algo} and \cref{HTK algo down} converge to the same solution, i.e. $\underline{V}=\overline{V}$. This leads us to think that the solution to \cref{Eq: D-HJB H} is in fact unique also with $0<\theta<1$, even if this is not yet proved.

 \section{Numerical results}\label{Sec: Numerical}
For a complete solution of the MFG system \cref{MFG}, including the computation of equilibrium interest rates, we refer the reader to \cite{achdou2026continuous}. Throughout the numerical experiments, we use the stopping tolerances $\mathrm{Tol}_c=10^{-7}$ and $\mathrm{Tol}_V=10^{-10}$. Unless otherwise specified, the model parameters are fixed at
$$
\rho=0.05,\quad y_1=0.5,\quad y_2=1.5,\quad
\lambda_1=\lambda_2=0.2,\quad
\ux=-0.15,\quad \bar{x}=40,
$$
$$
A=0.95,\quad \alpha=0.35,\quad \delta=0.1.
$$
The preference parameters $\gamma$ and $\psi$ vary across tests, as specified below. We use a baseline grid spacing of $\D x=0.01$, except in the convergence-rate experiment reported in \cref{Fig: rate}. We consider three MFG tests and report the corresponding equilibrium interest rates $r^*$. The HJB equations in {\textit{Test 1 MFG}} ($\theta>1$) are solved using \cref{HN algo}, while those in {\textit{Test 2 MFG}} and {\textit{Test 3 MFG}} ($0<\theta<1$) are solved using both \cref{HTK algo} and \cref{HTK algo down}.
\begin{itemize}
\item[] {\textit{Test 1 MFG:}} $\g=2$, $\psi=0.4$, $r^*=0.0288$.
\item[] {\textit{Test 2 MFG:}} $\g=4$, $\psi=0.5$, $r^*=0.0266$. {\textit{Test 3 MFG:}} $\g=20$, $\psi=0.5$, $r^*=0.0086$.
\end{itemize}

\begin{figure}[h!]
	\centering
	\caption{ Saving policy and asset distribution for Test 2 (solid) and 3 (dotted)}\label{Test 23 sg}
	\begin{tabular}{cc}
	\includegraphics[width=0.46\textwidth]{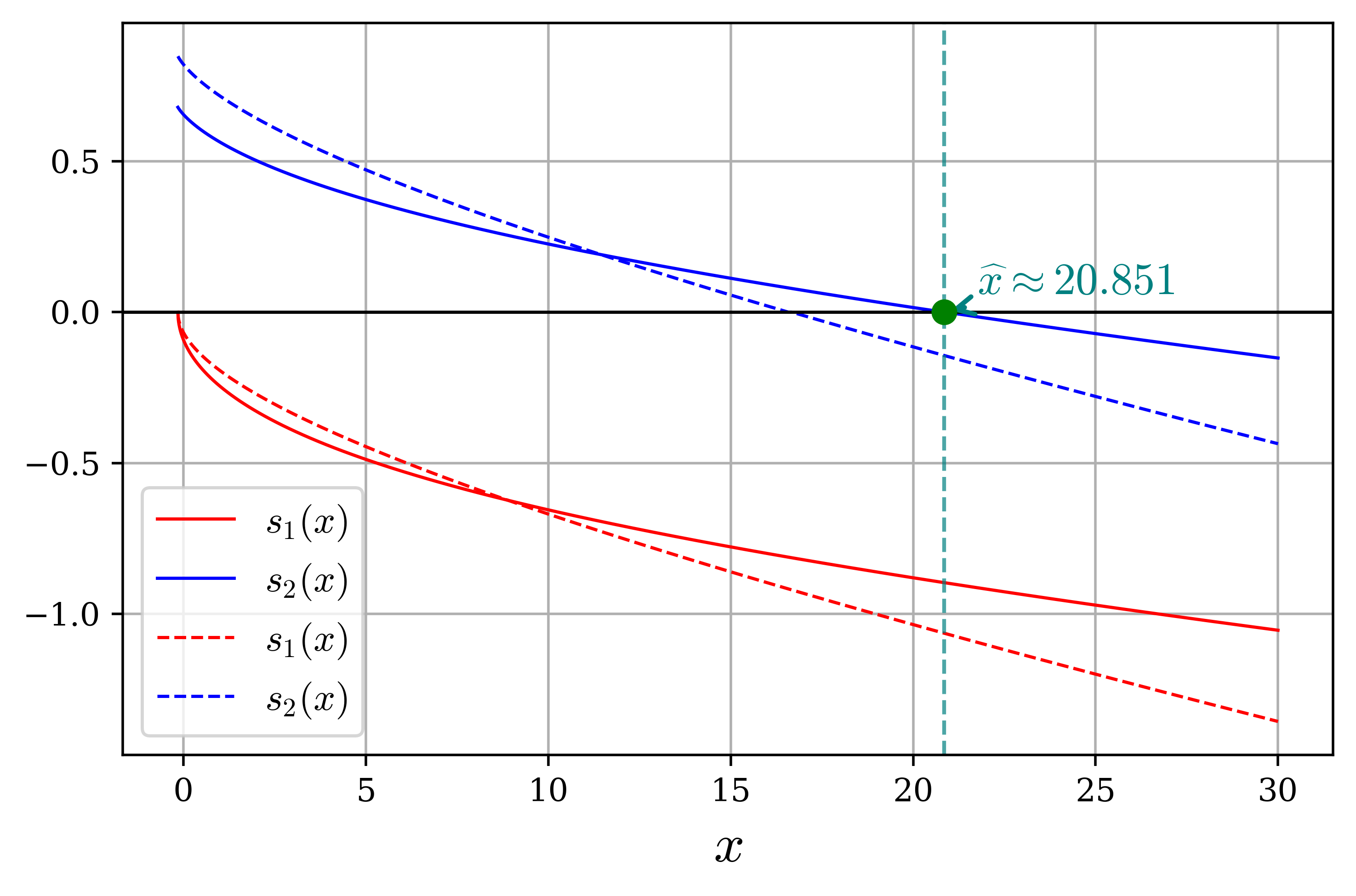} &
	\includegraphics[width=0.46\textwidth]{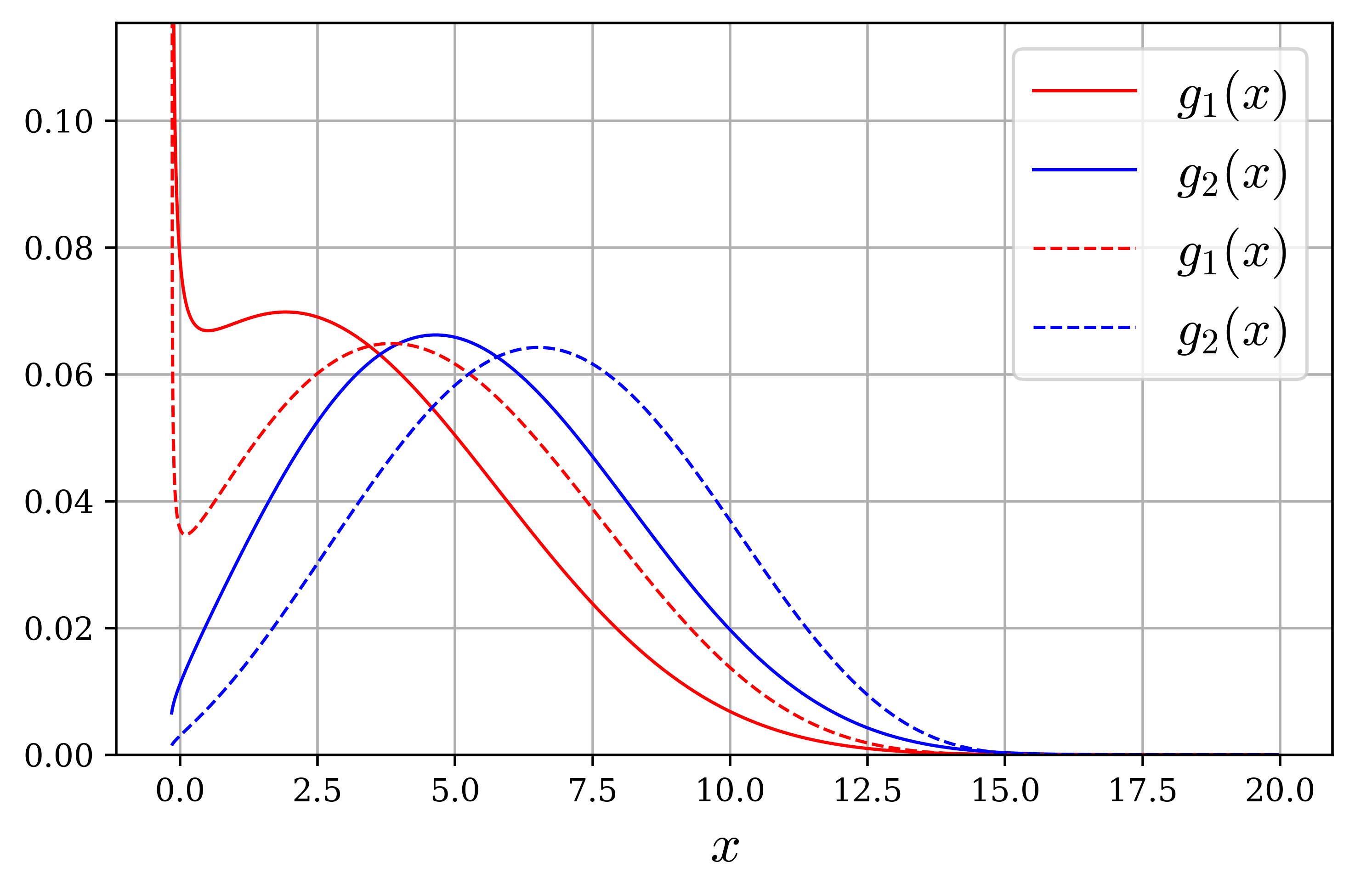} \\
	\end{tabular}
\end{figure}
Next, we focus on the convergence analysis of the finite-difference scheme and the algorithmic performance of the HJB solver. To isolate these aspects, tests for which the interest rates are held fixed at the equilibrium values are referred to as {\textit{Test 1 HJB}}, {\textit{Test 2 HJB}} and {\textit{Test 3 HJB}}.\par
For Test 1, \cref{Fig: vn_HN} illustrates the fact that the sequence produced by \cref{HN algo} is non decreasing as stated in \cref{Thm: Howard-Newton}, i.e. $V^{(n)} \preceq V^{(n+1)}$. For {\textit{Test 2 HJB}}, \cref{Fig: vk_HTK} (resp. \cref{Fig: vk_HTKdown})  illustrates the fact, stated in \cref{Thm: Tarski convergence} (resp \cref{Thm: Tarski convergence maximal}) , that \cref{HTK algo} (resp. \cref{HTK algo down}) produces a non decreasing (resp. non increasing) sequence of discrete functions.

 \begin{figure}[h!]
   \centering
    \caption{\cref{HN algo} constructs a non decreasing sequence of grid functions}\label{Fig: vn_HN}
    \includegraphics[width=0.5\linewidth]{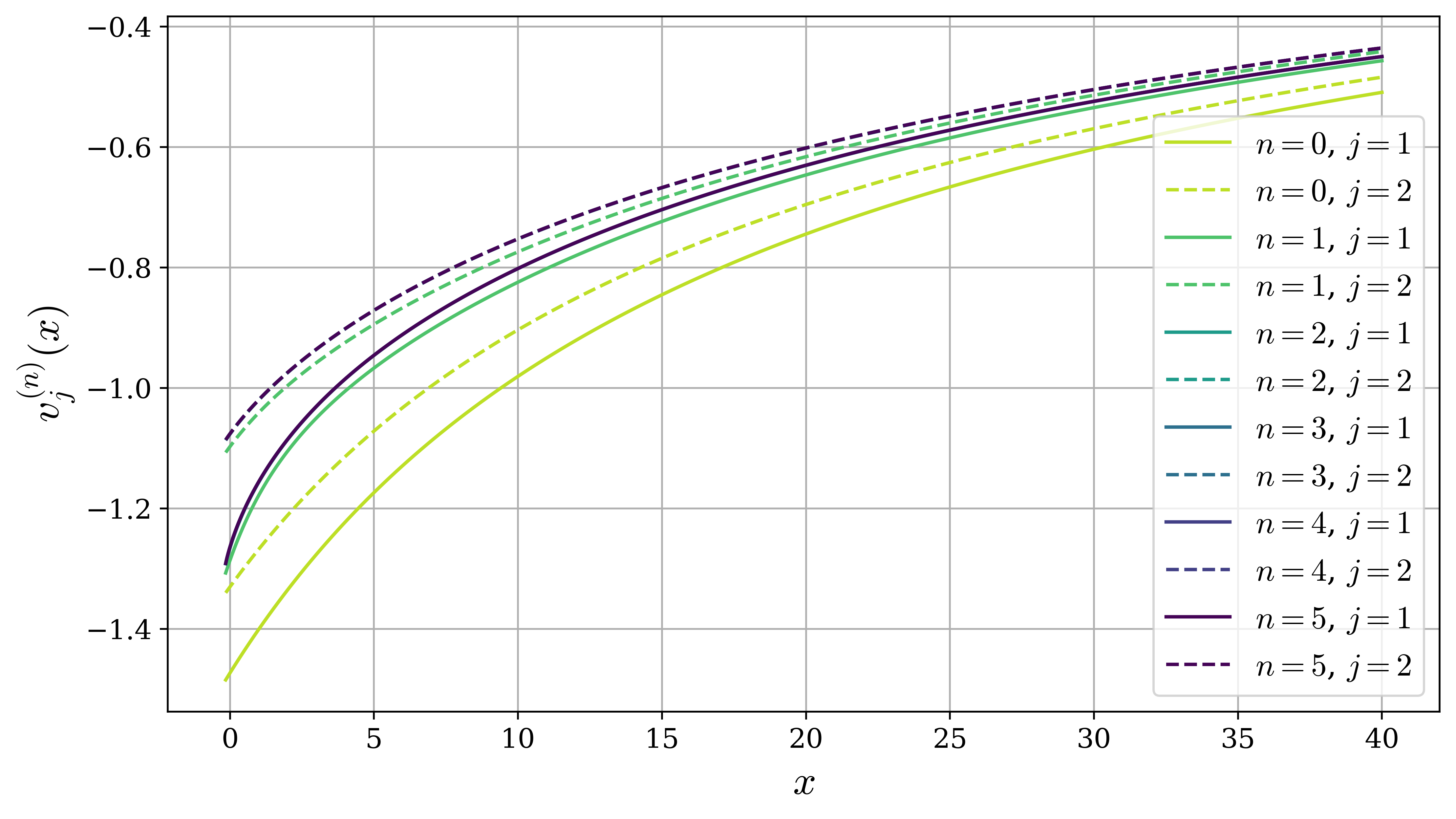}
\end{figure}
  \begin{figure}[h!]
   \centering
    \caption{\cref{HTK algo} constructs a non decreasing sequence of grid functions}\label{Fig: vk_HTK}
    \includegraphics[width=0.5\linewidth]{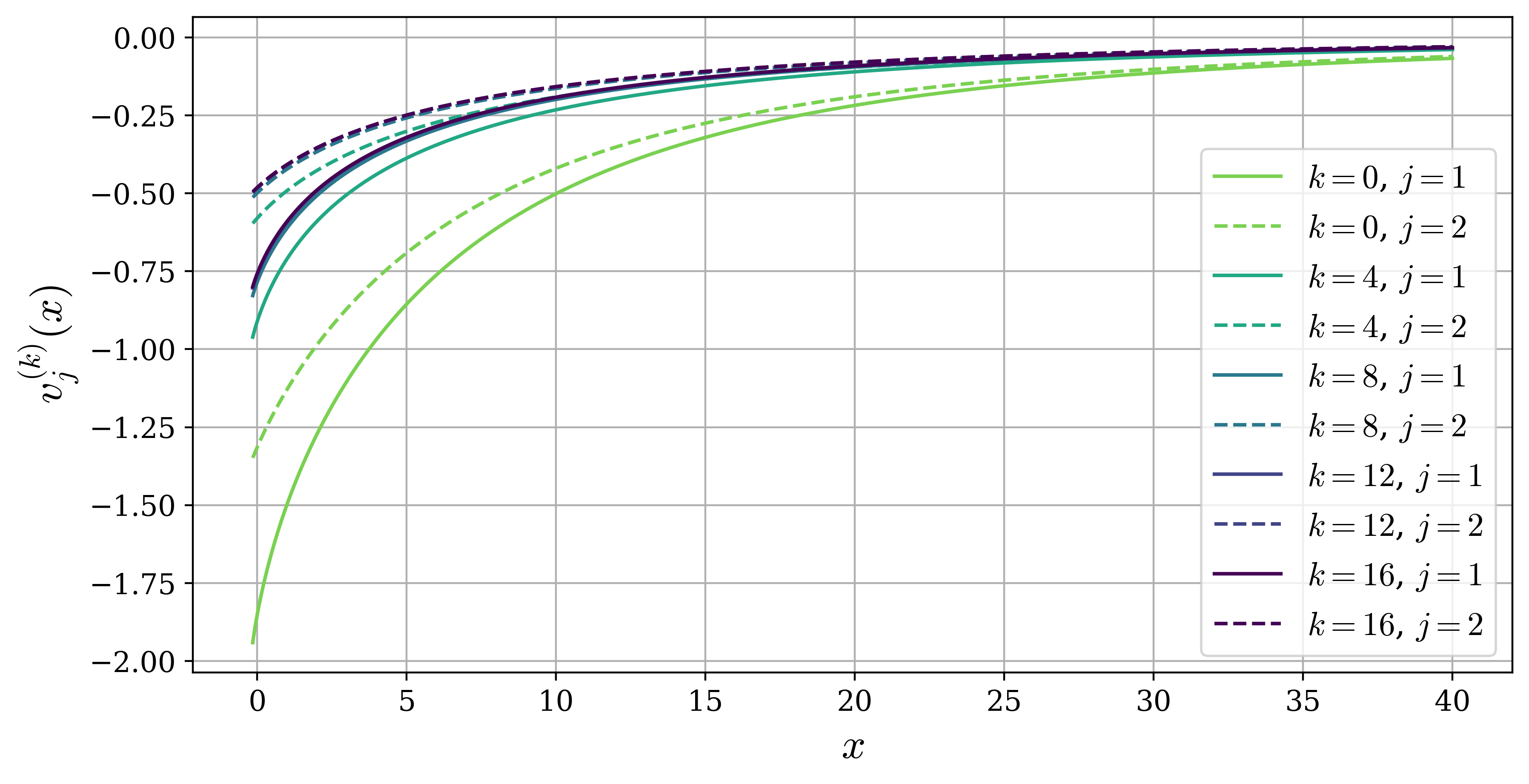}
\end{figure}
  \begin{figure}[h!]
   \centering
    \caption{\cref{HTK algo down} constructs a non increasing sequence of grid functions}\label{Fig: vk_HTKdown}
    \includegraphics[width=0.5\linewidth]{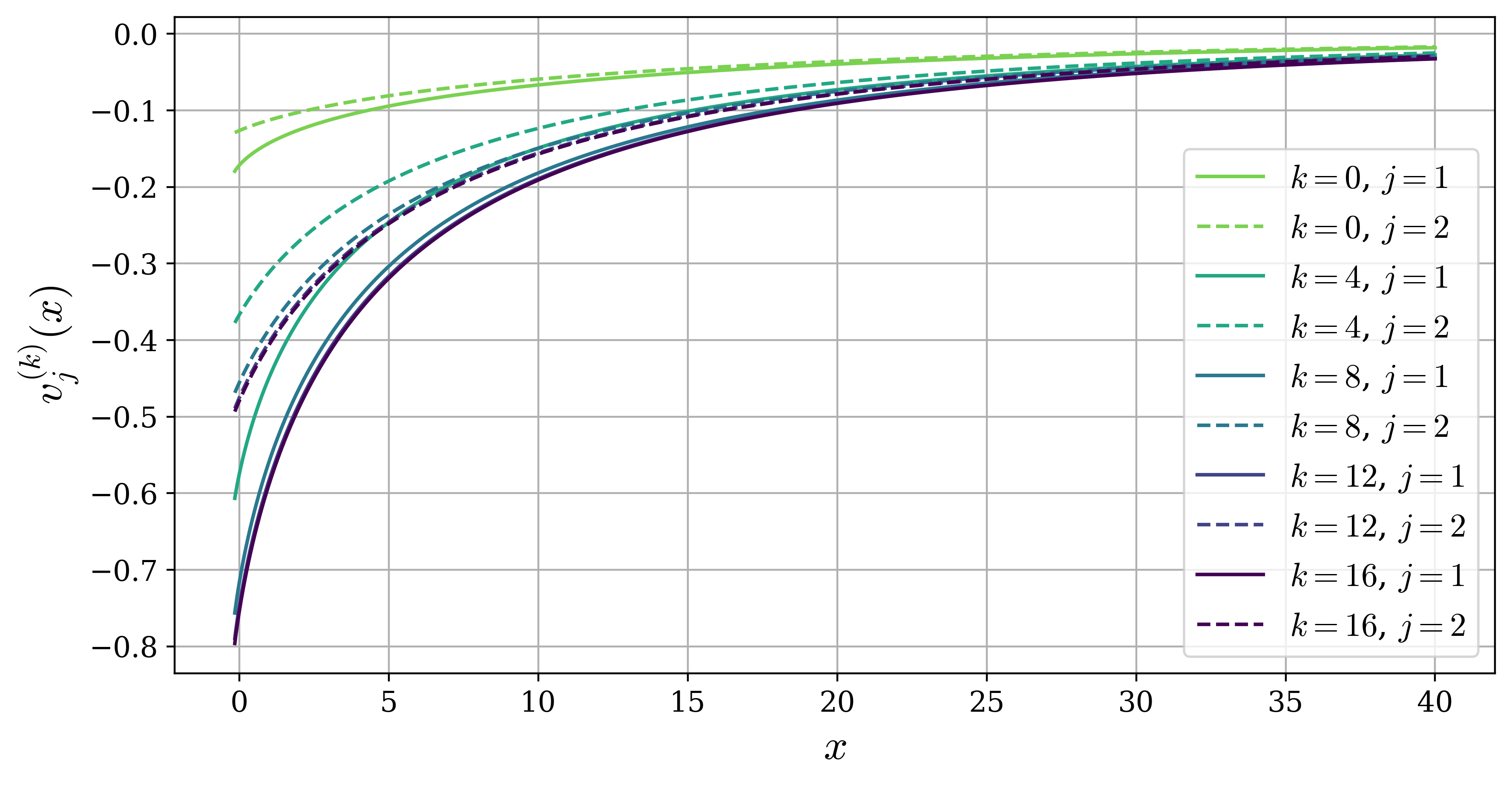}
\end{figure}

We now consider an additional {\textit{Test 4 HJB}} in which all parameters are identical to those in {\textit{Test 1 HJB}}, with the exception that $\lambda_2=0.002$. \cref{s14} illustrates the results of \cref{Prop: saving}. In particular,  we observe $s_2(\ux)=0$ in {\textit{Test 4 HJB}}. Such a behavior would not occur if \cref{rho-r2} is satisfied. In the case corresponding to {\textit{Test 4 HJB}}, the agents associated with $y_2$ do not have precautionary motive to save if the transition risk $\lambda_2$ is too small. Note that our scheme still works for {\textit{Test 4 HJB}}, although we have not analyzed this situation for brevity. \par
\begin{figure}[h!]
   \centering
    \caption{Saving policy in {\textit{Test 1 HJB}} (solid) and {\textit{Test 4 HJB}} (dotted)}\label{s14}
    \includegraphics[width=0.46\linewidth]{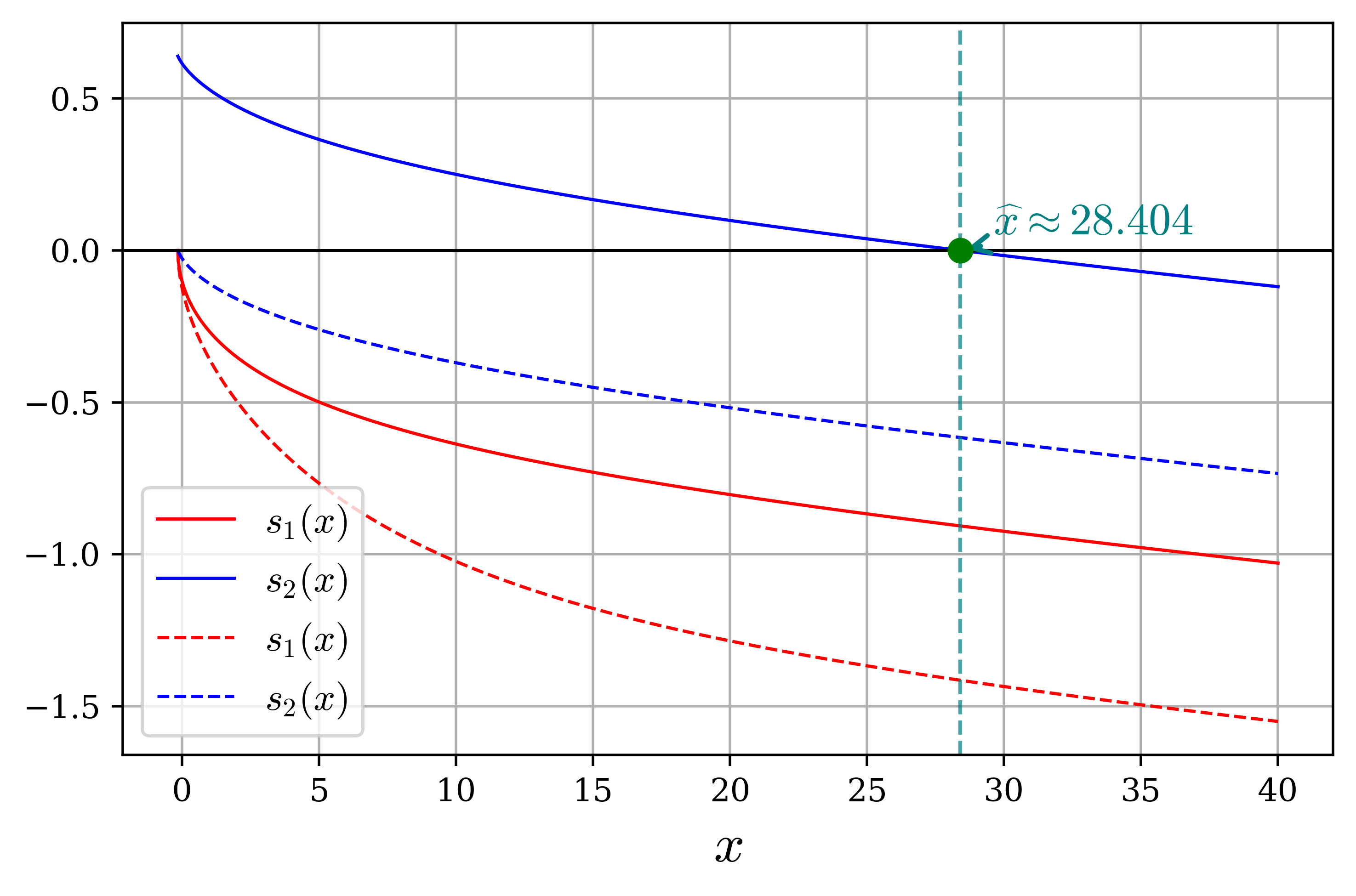}
\end{figure}
We next supplement the numerical results of {\textit{Test 1 HJB}} with a numerical estimation of the convergence rate illustrating  \cref{Subsec: rate}. Since the exact solution $(v_1, v_2)$ of the HJB system is not available in closed form,
we compute a reference solution $(V_1^{\text{ref}}, V_2^{\text{ref}})$ on a fine grid with
$\Delta x_{\text{ref}} = 0.001$ and treat it as a proxy for the true solution. We then solve the scheme on a sequence of ten coarser grids, with $\Delta x$ ranging from $0.002$ to $0.1$, spaced geometrically so that the grid spacings are uniformly distributed
on a log scale. For each coarse grid spacing $\Delta x > \Delta x_{\text{ref}}$, the reference
solution is interpolated onto the coarse grid $\mathcal{G}^{\Delta x}$ via the piecewise linear
interpolation $\mathbb{I}[\cdot]$. We use $e_j^{\D x}=\max_{x\in \mathcal{G}^{\D x}} \left\vert \mathbb{I}[V_j^{\text{ref}}](x)-V_{j}(x)\right\vert$ to substitute $\max_{x\in \mathcal{G}^{\D x}} \left\vert v_j(x)-V_{j}(x)\right\vert$. \par
\cref{Fig: rate} displays errors as functions of
$\Delta x$ on a log-log scale, together with a reference line of slope $1$. Both error curves decrease monotonically as $\Delta x \to 0$, confirming the convergence of the
scheme. Moreover, the curves run parallel to the slope-$1$ reference line, consistent with the theoretical bound established in \cref{Thm: convergence rate}.
\begin{figure}[h!]
   \centering
    \caption{convergence rate in {\textit{Test 1 HJB}}}\label{Fig: rate}
    \includegraphics[width=0.5\linewidth]{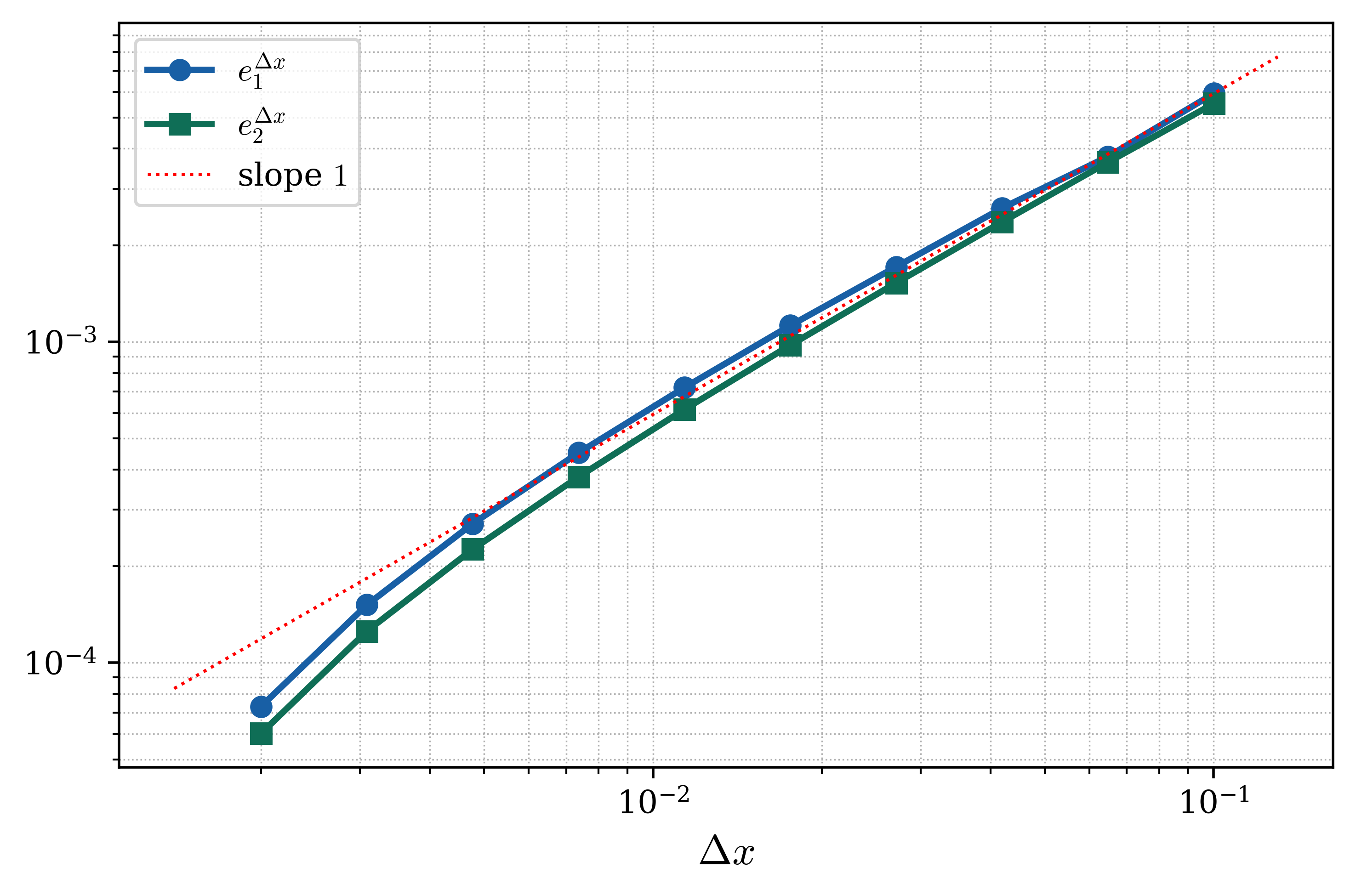}
\end{figure}\\
\noindent{\bf{Acknowledgment.}} \small{Yves Achdou was  partially supported by the ANR (Agence Nationale de la Recherche) through project ANR-22-CE40-0010-01 and  by the chair Finance and Sustainable Development and FiME Lab (Institut Europlace de Finance).\par
Part of this paper was written while Qing Tang was a visitor at the Laboratoire Jacques-Louis Lions, Universit\'e Paris Cit\'e, supported by a China Scholarship Council grant No. 202406410221.}

	\end{document}